\documentclass[12pt]{article}
\usepackage{bibunits}

\usepackage{hyperref}

\newcommand{\note}[1]{}

\renewcommand\marginpar[1]{}

\usepackage{amsmath} 
\usepackage{amssymb}
\newcommand{\average}{-\!\!\!\!\!\!\int}
 
\newcommand{\bigfrac}[2]{ { \displaystyle \frac{#1}{#2}}}
\newcommand{\comment}[1]{}

\renewcommand{\div}{{\mathop{\rm div}\nolimits}}
\newcommand{\diam}{\mathop{\rm diam}\nolimits}
\newcommand{\dist}{\mathop{\rm dist}\nolimits}

\providecommand{\qed}{\vrule height 6pt depth 0pt width 3 pt}

\newcommand{\reals}{{\bf R}}
\renewcommand{\Re}{\mathop{\rm Re}\nolimits}

\newcommand{\supp}{\mathop{\rm supp}\nolimits }

\newcommand{\BibTeX}{{\rm B\kern-.05em{\sc i\kern-.025em b}\kern-.08em     
    T\kern-.1667em\lower.7ex\hbox{E}\kern-.125emX}}

\newcommand{\pnabla}{\nabla'} 
\newcommand{\crease}{\Lambda}

\newcommand{\sring}[1]{\Sigma_{#1}}
\newcommand{\funsol}{\Xi}

\newcommand{\ball}[2]{B_{#2}(#1)}
\newcommand{\sball}[2]{\Delta_{#2}(#1)}
\newcommand{\cyl}[2]{Z_{#2}(#1)}
\newcommand{\locdom}[2]{\Omega_{#2}(#1)}

\newcommand{\dball}[2]{\Psi _{ #2}(#1)}
\newcommand{\nontan}[1]{#1^*}
\newcommand{\bdry}{\partial}
\newcommand{\ntar}[1]{\Gamma(#1)}

\newcommand{\ellipconst}{\lambda}

\newcommand{\sobolev}[2]{W^{#1,#2}}
\providecommand{\tangrad}{\nabla _t}

\newenvironment{proof}[1][Proof]{\begin{trivlist}\item[\hskip \labelsep
{\it #1. }]}{\hfill \qed \goodbreak \end{trivlist}}

\numberwithin{equation}{section} 
\newtheorem{theorem}[equation]{Theorem}

\newtheorem{lemma}[equation]{Lemma}
\renewcommand{\thetheorem}{\arabic{section}.\arabic{theorem}}
\renewcommand{\theequation}{\arabic{section}.\arabic{equation}}

\begin{document}
\begin{bibunit}
\title{The mixed problem for the Laplacian in Lipschitz domains}

\author{
Katharine A. Ott
\footnote{Research supported, in part, by the National
Science Foundation.} 
\\ Department of Mathematics 
\\University of Kentucky
\\ Lexington, Kentucky 
\and 
Russell M. Brown
\\ Department of Mathematics 
\\University of Kentucky
\\ Lexington, Kentucky }

\date{}

\maketitle

\abstract{ 
  We consider the mixed boundary value problem, or Zaremba's
  problem, for the Laplacian in a bounded Lipschitz domain $\Omega$ in
  $ \reals ^ n$, $n\geq 2$. We decompose the boundary $ \partial
  \Omega= D\cup N$ with $D$ and $N$ disjoint. The boundary between $D$
  and $N$ is assumed to be a Lipschitz surface in $\partial
  \Omega$. We find an exponent $q_0>1$  so that for $ p $ between $ 1 $
  and $q_0$ we may solve the mixed problem for $L^p$. Thus, if we
  specify Dirichlet data on $D$ in the Sobolev space $\sobolev 1 p
  (D)$ and Neumann data on $N$ in $ L^ p (N)$, the mixed problem with
  data $f_N$ and $f_D$ has a unique solution and the non-tangential
  maximal function of the gradient lies in $L^p( \partial \Omega)$.
  We also obtain results for $p=1$ when the data comes from Hardy
  spaces.


}

\section{Introduction}

Over the past thirty years, there has been a great deal of interest in
studying boundary value problems for the Laplacian in Lipschitz
domains.  A fundamental paper of Dahlberg \cite{BD:1977} treated the
Dirichlet problem. Jerison and Kenig \cite{JK:1982c} treated the
Neumann problem and provided a regularity result for the Dirichlet
problem.  Another boundary value problem of interest is the mixed
problem or Zaremba's problem  where we specify Dirichlet
data on part of the boundary and Neumann data on the remainder of the
boundary. 
To state this boundary value problem, we let $ \Omega$ be a bounded open set in $ \reals ^n$ and suppose
that we have written $ \partial \Omega = D\cup N$ where $D$ is an open
subset of the boundary and $ N =\partial \Omega \setminus D$.  
We consider the {\em $L^p$-mixed problem }  by which we mean the
boundary value problem
\begin{equation} \label{MP}
\left\{
\begin{array}{ll}
\Delta u = 0, \qquad & \mbox{in }  \Omega\\
u = f_D, \qquad &  \mbox{on }  D\\
\bigfrac { \partial u }{ \partial \nu} = f_N,\qquad  & \mbox{on } N \\
\nontan{(\nabla u ) } \in L^ p ( \partial \Omega).  
\end{array}
\right.
\end{equation}  
Here, we are using $ \nontan {(\nabla u)}$ to denote the non-tangential maximal
 function of  $\nabla u$ and the restriction  to the 
boundary  of $u$ and $\nabla u$ 
are defined using non-tangential limits. See section \ref{Definitions}
for details. 
  The normal derivative at the boundary $ \partial u
/\partial \nu$ is defined as $ \nabla u \cdot \nu$ where $\nu $ is the
outer unit normal defined a.e.~on the boundary. 
Our goals are to find conditions on $\Omega$, $N$ and $D$ which
allow us to show that   (\ref{MP}) has at most one solution and 
to find conditions on $ \Omega$, $N$,  $D$,  $f_N$ and $f_D$ which guarantee the existence of
solutions. 

The study of  the mixed problem in Lipschitz
domains is  listed as an open problem in  Kenig's CBMS lecture notes
\cite[Problem 3.2.15]{CK:1994}.  Recall that simple examples show that we cannot
expect to find solutions whose  gradient   lies in $L^2$ of the
boundary. For example, the function $ \Re \sqrt z$ on the upper
half-plane has zero Neumann data on the positive real axis and zero
Dirichlet data on the negative real axis but the gradient is not
locally square integrable on the boundary of the upper half-space.
This appears to present a technical problem as the standard
technique for studying boundary value problems has been the Rellich
identity which produces estimates in $L^2$.

In 1994, one of the authors 
observed that the Rellich identity could
be used to study the mixed problem in a restricted class of Lipschitz
domains \cite{RB:1994b}. Roughly speaking, this work requires that the
sets $N$ and $D$ meet at an angle less than $\pi$. Based on this work
and the methods used by Dahlberg and Kenig 
 to study the Neumann problem \cite{DK:1987}, J. Sykes
\cite{JS:1999,SB:2001} established results for the mixed problem in
a restricted class of Lipschitz graph domains.  I.~Mitrea and M.~Mitrea \cite{MM:2007} have
studied the mixed problem for the Laplacian with data taken from a
large family of function spaces, but with a restriction on the class
of domains.   Brown  and I.~Mitrea have studied the
mixed problem for 
the Lam\'e system \cite{MR2503013} and  Brown, I.~Mitrea,
M.~Mitrea and Wright have considered a  mixed problem for the Stokes system
\cite{MR2563727}.  More recently, Lanzani, Capogna and Brown
\cite{LCB:2008} used a variant of the Rellich identity
 to establish an estimate for the mixed
problem in two-dimensional graph domains when the data comes from
weighted $L^2$ spaces and the Lipschitz constant is less than one.
The present work also relies on weighted estimates, but uses a
simpler, more flexible approach that applies to all Lipschitz domains.

Several other authors have treated the mixed problem in various
settings.  Verchota and Venouziou \cite{MR2500502} treat a large class
of three dimensional polyhedral domains under the condition that the
Neumann and Dirichlet faces meet at an angle of less than $\pi$.
Maz'ya and Rossman \cite{MR:2005,MR:2007,MR:2006} have studied the
Stokes system in polyhedral domains.  Finally, we note that Savar\'e
\cite{GS:1997} has shown that on smooth domains, we may find solutions
in the Besov space $ B^ {2,\infty}_{ 3/2}$.  This result seems to be
very close to optimal. The example $ \Re \sqrt z$ described above
shows that we cannot hope to obtain an estimate in the Besov space $
B^ {2,2}_{3/2}$.

We outline the rest of the paper and describe the main tools of the
proof. Our first main result is an existence result for the mixed
problem when the Neumann data is an atom for a Hardy space. We begin
with the weak solution of the mixed problem and use Jerison and
Kenig's results for the Dirichlet problem and Neumann problem
\cite{JK:1982c} to obtain estimates for the gradient of the solution
on the interior of $D$ or $N$. This leads to a weighted estimate where
the weight is a power of the distance to the common boundary between
$D$ and $N$. The estimate involves a term in the interior of the
domain $\Omega$. We handle this term by showing that the gradient of a
weak solution lies in $L^ p (\Omega)$ for some $p>2$.  The
$L^p(\Omega)$ estimates for the gradient of a weak solution are proved
in section \ref{Reverse} using the reverse H\"older technique of
Gehring \cite{FG:1973} and Giaquinta and Modica \cite{MR549962}.
Using this weighted estimate for solutions of the mixed problem, we
obtain existence for solutions with Hardy space data by extending the
methods of Dahlberg and Kenig \cite{DK:1987}. Uniqueness of solutions
is proven in section \ref{Unique}.

With the Hardy space results in hand, we establish the existence of
solutions to the mixed problem when the Neumann data is in $L^p ( N)$
and the Dirichlet data is in the Sobolev space $ \sobolev 1 p
(D)$. This is done in sections \ref{BoundaryReverse} and
\ref{LpSection} by adapting the reverse H\"older technique used by
Shen to study boundary value problems for elliptic systems
\cite{ZS:2007}. The novel feature in our work is that we are able to
use the estimates in Hardy spaces proven in section \ref{Atoms},
whereas Shen's work begins with existence in $L^2$.

\section{Definitions and preliminaries} 
\label{Definitions}

We say that a bounded, connected open set $\Omega$ is a Lipschitz
domain if the boundary is locally the graph of a Lipschitz function.
To make this precise, for $M>0$, $ x\in \partial \Omega $ and $ r>0$,
we define a {\em coordinate cylinder} $\cyl x r $ to be $\cyl x r = \{
y : |y'-x'|< r , \ |y_n -x_n | < ( 1+M)r \}$.  We use coordinates
$(x', x_n ) = ( x_1, x'', x_n ) \in \reals \times \reals ^ { n- 2 }
\times \reals$ and assume that this coordinate system is a translation
and rotation of the standard coordinates.  We say that $ \Omega$ is a
{\em Lipschitz domain } if for each $x$ in $\partial \Omega$, we may
find a coordinate cylinder and a Lipschitz function $\phi : \reals ^ {
  n-1} \rightarrow \reals$ with Lipschitz constant $M$ so that
\begin{eqnarray*} 
\Omega \cap \cyl x r &  =&  \{ (y', y_n ) : y_ n > \phi (y') \} \cap
\cyl x r  \\
\partial \Omega \cap \cyl x r & = & 
 \{ (y', y_n ) : y_ n = \phi (y') \}\cap \cyl x r .  
\end{eqnarray*}

For a Lipschitz domain $ \Omega$, we define a {\em decomposition of the
boundary  for the mixed problem}, $\partial \Omega = D \cup N$,  as
follows.   We assume that $D$ is a relatively 
open subset of $ \partial \Omega$, $N= \partial \Omega \setminus D$
and let $\crease  $ be the boundary 
(relative to $\partial \Omega$) of $D$. For each $x$ in $ \crease$, we
require that a coordinate cylinder centered at $x$ have some
additional properties. 
We ask that there be a coordinate system $(x_1,
x'', x_n)$, a coordinate cylinder $\cyl x r $, a function $\phi$ as
above and also a Lipschitz function $\psi: \reals^{ n-2} \rightarrow \reals$
with  Lipschitz constant $M$  so that
\begin{eqnarray*} 
\cyl x r \cap D  & = &  \{ (y_1, y'', y_n ) : y_ 1 > \psi
(y''), \ y_n = \phi(y') \} \cap \cyl x r \\
\cyl x r \cap N  & = &  \{ (y_1, y'', y_n ) : y_ 1 \leq  \psi (y''),
\ y_n = \phi(y') \}  \cap \cyl x r . 
\end{eqnarray*}

We fix a covering of the boundary by coordinate cylinders $\{ \cyl
{x_i } { r _i } \}_{i=1} ^ L$ so that each $\cyl {x_i } { 100r _i } $
is also a coordinate cylinder. We assume that for each $i$, the
cylinder $\cyl {x_i}  { 100r_i}  \cap \partial \Omega \subset D$,
$\cyl {x_i }{
  100r_i}  \cap \partial \Omega \subset N $ or $\cyl {x_i} {100r_i}$
is one the coordinate cylinders from the definition of the boundary
decomposition for the mixed problem.  We let $r_0 = \min \{ r_i : i
=1, \dots ,L\}$ be the smallest radius in the collection.

We will call a Lipschitz domain $ \Omega$ and a decomposition of the
boundary $ \partial \Omega = N \cup D$ satisfying the above properties
a {\em standard domain for the mixed problem}. 

We will use $ \delta (y) = \dist (y, \crease)$ to denote the distance
from a point $y $ to $\crease$. 
We will let $ \ball x r = \{ y : |y-x| < r \}$ denote the standard
ball in $ \reals ^n$ and then $ \sball x r = \ball x r \cap \partial
\Omega$ will denote a {\em surface ball}. Throughout this
paper we will need to be careful of several points. The surface balls
may not be connected and we will  use the notation $ \sball x r $
where $ x$ may not be on  the boundary.  We  use $ \dball x
r $ to stand for $ \ball x r \cap \Omega$. 
Since $\crease$ is a Lipschitz graph, we may find  a constant $c =
c(n,M) >0$
so that we have the property 
\begin{equation} \label{SurfProp}
\mbox{If $ x\in \crease$ and $ 0<r< r_0$, then $\sigma ( \sball x {r}
  \cap D  )  > c  r^ { n-1}$. } 
\end{equation} 
Here and throughout this paper, we use $ \sigma$ for surface measure. 

Our main tool for estimating solutions will be the non-tangential
maximal function. We fix $ \alpha > 0$ and for  $ x\in \partial
\Omega$  we
define a {\em non-tangential approach region } by
$$
\ntar x = \{ y \in  \Omega : |x-y | \leq ( 1+ \alpha)  \dist (y, \partial
\Omega) \}. $$ 
Given a function $u$ defined on $ \Omega$, we define the 
{\em  non-tangential maximal function }  by 
$$
\nontan{u} (x) = \sup _{ y \in \ntar x } |u (y) |, \qquad x \in
\partial \Omega.
$$
It is well-known that for different values of $ \alpha$, the
non-tangential maximal functions have comparable $L^p$-norms. Thus,
the dependence on $\alpha $ is not important for our purposes and we
suppress the value of $\alpha $ in our notation. 
In (\ref{MP}), we define the restriction of $u$ and $ \nabla u$ to the
boundary using 
non-tangential limits. Thus, for a function $v$ defined in $ 
\Omega$ and $ x \in \partial \Omega$, we define 
$$
v(x) = \lim _{ \Gamma (x) \ni y \rightarrow x } v(y)
$$ 
provided the limit exists. It is well-known that if $v$ is harmonic in
a Lipschitz domain, then the non-tangential limits exist at almost
every point where the non-tangential maximal function is finite. In
addition, if the non-tangential maximal function of  $ \nabla u$ lies
in $L^p( \partial \Omega)$, then according to the argument in
\cite[Lemma 2.2]{RB:1995a},  as
corrected in Wright  \cite{MR2713677}, the non-tangential maximal
function of $u$ lies in an $L^p$-space and hence has non-tangential limits
a.e.. 

Many of our estimates will be of a local, scale invariant form and
hold on a scale $r$ that is less than $r_0$. The constants in these
local estimates will depend on the constant $M$, the dimension $n$,
and any $L^p$-indices that appear in the estimate.  If a constant
depends on  $M$, $n$, any $L^p$-indices and also depends on the collection
of coordinate cylinders which cover $ \partial \Omega$ and the
constant in the coercivity condition (\ref{coerce}), then we say that the constant depends on
the global character of $ \Omega$, $N$ and $D$.

We will use $L^p(E)$ to denote $L^p$-spaces. If $ E\subset \partial
\Omega$, then we use the $(n-1)$-dimensional measure on the boundary
to define the $L^p$-space. Otherwise, the $L^p$-norm is taken with
respect to $n$-dimensional Lebesgue measure. For $ \Omega$ an open
subset of $ \reals ^n$, $k=1,2,\dots$ and $ 1\leq p \leq \infty$, we
use $ \sobolev kp(\Omega)$ to denote the Sobolev space of functions
having $k$ derivatives in $L^p( \Omega) $.  We introduce notation for
the tangential gradient of a function defined on the boundary, $
\tangrad u$. If $u$ is a smooth function defined in a neighborhood of
$ \partial \Omega$, then we have that $\tangrad u = \nabla u - (\nabla
u \cdot \nu )\nu$. See \cite[p.~580]{GV:1984} for more details.  For 
 $D$ an open subset of $\partial \Omega$, we use $ \sobolev 1
p(D)$ to denote the Sobolev space of functions defined on $D$ and
having one derivative in $L^p(D)$. The norm in this space is given by $
\|f\|_{ \sobolev 1 p (D)}=  \| f\|_{ L^ p (D)}+\|\tangrad f \|_ { L^ p (D)}$.

Before stating the main theorem, we recall the definitions of atoms and
atomic Hardy spaces.   
We say that $a$ is an {\em atom  for the boundary }$ \partial \Omega$ if 
$a$ is supported in a surface ball $ \Delta _r (x)$ for some $x$ in
$ \partial \Omega$, $\|a\|_{L^\infty(\partial \Omega)}  \leq
1/\sigma(\Delta _r(x))$ and $ 
\int _{ \partial \Omega } a \, d\sigma = 0. $  

When we consider the mixed problem, we will want to consider atoms for
the subset $N$. We say that {\em   $ a$ is an atom for $N$}  if $a$ is the
restriction to $N$ of a function $\tilde a $ which is an atom for $ \partial
\Omega$. For   $ N $ a subset of  $ \partial \Omega$, the Hardy
space $ H^1( N)$ is the collection of 
functions $f$ which can be represented as $ \sum \lambda_j a_j$ where
each $a_j$ is an atom for $N$   and the coefficients
satisfy  $\sum | \lambda _j|< \infty $. This includes, of course, the
case where $N = \partial \Omega$ and then  we obtain the standard
definition.  It is easy to see that the Hardy space $H^1(N)$ is the
restriction to $N$ of elements of the Hardy space $ H^ 1 ( \partial
\Omega)$.

We give a similar definition for the Hardy-Sobolev space $H^ { 1,1}$. 
 We say that $ A$ is an {\em
  atom for $H^ { 1,1 } ( \partial \Omega)$   } 
if $A$ is supported in a surface ball $ \sball x r $ for some  $x \in
\partial \Omega$  and 
$\| \nabla _t A \| _ { L^ \infty ( \partial \Omega ) } \leq 1/\sigma (
\sball x r )$. We say that $A$ is an {\em atom for $H^ { 1 ,1 } (D) $
}  
  if $A$ is the restriction to $D$ of an atom   $ \tilde A$ for $ \partial
  \Omega$. Again, the space  $H^ { 1,1 }( D )$ is the
  collection generated by taking sums of  $H^ {1,1}(D) $ atoms with
  coefficients in $   \ell ^1$. 
See the article of Coifman and Weiss  \cite{CW:1976} for more information
about Hardy spaces.

We are now ready to state our main theorem. 

\begin{theorem} \label{main} 
Let $ \Omega$, $N$ and $D$ be a  standard  domain for the mixed problem. 

a) For $ p \geq 1$, the $L^p$-mixed problem has at most one solution.

b) If $f_N$ lies in $H^ 1(N)$ and $f _D $ lies in $H^{1,1}(D)$, the
$L^1$-mixed problem has a solution which satisfies the estimate 
$$
\| (\nabla u )^*\| _ { L^ 1 (\partial \Omega )} \leq C ( \| f_N\|_ { H
  ^ 1( N) } + \|f_D\| _ { H ^ { 1,1 } (D)}) . 
$$

c) There exists $q_0  > 1 $, depending only on $M$ and $n$ so that for
$p$ satisfying $1< p< q_0$, we have: 
If $f_N \in L^ p (N)$ and $ f_D \in \sobolev 1 p(D)$, then the $L^p$-mixed
problem  has a solution $u$ which satisfies 
$$
\| (\nabla u ) ^ * \| _ { L^ p ( \partial \Omega ) }
\leq C ( \| f_N \| _{L^ p(N) } + \|f_D \| _ { \sobolev 1 p (D)} ) .
$$

The constants in the estimates depend on the global character of the
domain and the index $p$. 
\end{theorem}

The rest of the paper is devoted to the proof of this theorem. We outline
 the main steps of the proof. 
\begin{proof}[Outline of the proof] 
We begin by recalling that for the  Dirichlet problem with
data from a Sobolev space, we  obtain non-tangential maximal
function estimates for the gradient of the solution. This is treated
for $p=2 $ by 
Jerison and Kenig \cite{JK:1982c}  and  for  $ 1< p <  2$
by Verchota \cite{GV:1982,GV:1984}.  The Hardy space problem was
studied by
Dahlberg and Kenig \cite{DK:1987} and by D.~Mitrea in two dimensions
\cite[Theorem 3.6]{MR1883390}. 
Using these results,  it suffices to
prove Theorem \ref{main} in the case when the Dirichlet data is zero.

The existence when the  Neumann data is taken from the atomic Hardy
space  and
the Dirichlet data is  zero is given in
Theorem \ref{HardyTheorem}. The existence for  $L^p$ data appears in section
\ref{LpSection}. It suffices to establish uniqueness when $p=1$ and
this is treated in Theorem \ref{uRuniq}.
\end{proof} 

\section{Higher integrability of the gradient of a weak solution} 
\label{Reverse}

It is well-known that one can obtain higher integrability of the
gradient of weak solutions of an elliptic equation. An early result of
this type is due to Meyers \cite{NM:1963}. Meyers's result has been
extended to the mixed problem by Gr\"oger \cite{MR990595}. However, we
choose to obtain our estimates using the reverse H\"older technique
introduced by Gehring \cite{FG:1973} and Giaquinta and Modica
\cite{MR549962} (we use the formulation from Giaquinta
\cite[p.~122]{MG:1983}).  This approach allows us to include non-zero
boundary data and obtain local, scale-invariant results.  At a few
points of the proof it will be simpler if we are working in a
coordinate cylinder $Z$ where we have that $ \partial \Omega\cap Z $
lies in a hyperplane. Thus, we will establish results for divergence
form elliptic operators with bounded measurable coefficients as this
class is preserved by a change of variable that will flatten part of
the boundary.

We will consider several formulations of the mixed problem. 
Our
goal is to obtain solutions whose  gradient lies in  $L^p( \partial
\Omega)$ for $p$ near 1. Our
argument begins with a weak solution whose 
gradient lies in $L^2 ( \Omega)$. We will show that under appropriate
assumptions on the data, this solution will have a gradient in $L^p (
\partial \Omega)$.

We describe a weak formulation of the mixed boundary value problem. 
  Some of the results of this section  will hold for solutions of
divergence form operators. Thus, we define weak solutions in this more
general setting.
For $D$ a subset of the boundary, we let
$W^{1,2}_D(\Omega)$ be the closure in $W^{1,2}(\Omega)$ of functions
in $ C_0^ \infty ( \reals ^ n )$  for which $ \supp u \cap \bar D =
\emptyset$. 
We let
 $ W_D^{1/2,2}(\partial \Omega) $ be the restrictions to $ \partial
\Omega$ of the space $ W^ {1,2}_D( \Omega)$. We define $W^{-1/2,2}_D(
\partial \Omega)$ to be the dual of $W^ { 1/2,2} _D( \partial
\Omega)$.  The Neumann data $f_N$ will be taken from the space
$W^{-1/2,2}_D(\partial \Omega)$.  If $A(x) $ is a symmetric matrix
with bounded, measurable entries and satisfies the ellipticity
condition $\ellipconst   |\xi |^2 \geq A(x) \xi \cdot \xi \geq
\ellipconst^{-1}  |\xi |^ 2
$ for some $  \ellipconst  >0$ and all $ \xi \in \reals ^ n$, we consider
the problem
\begin{equation}\label{WeakMix}
\left \{ 
\begin{array}{ll} 
\div A \nabla u = 0 , \qquad &  \mbox{in } \Omega \\
u =  0 , \qquad &  \mbox{on } D \\
A\nabla u \cdot \nu = f _N,  &  \mbox{on }N. 
\end{array}
\right. 
\end{equation}
We say that $u$ is a {\em weak solution }of this problem if $u \in W^
{ 1,2 } _D ( \Omega) $ and we have
$$
\int _ \Omega A  \nabla u \cdot \nabla v \, dy  =  \langle f _N ,
v \rangle _ { \partial \Omega}  , \qquad \mbox{for all } v \in W^ { 1,2 }_D ( \Omega).
$$
Here,  we are using $ \langle \cdot , \cdot \rangle_{ \partial \Omega}$
to denote the pairing between $ W^ { 1/2,2 } _D( \partial \Omega)$ and
the dual $ W^ {- 1/2,2 } _D( \partial \Omega)$. 
To establish existence of weak  solutions of the mixed problem, we
assume the  coercivity  condition 
\begin{equation} \label{coerce}
\|u\|_{ L^2 ( \Omega) } \leq c \|\nabla u \|_{ L^2( \Omega)}, \qquad u
\in W^ { 1 , 2 } _D ( \Omega) .
\end{equation}
Under this assumption, the existence and uniqueness of weak solutions
to (\ref{WeakMix}) is a consequence of the Lax-Milgram theorem.  It is
easy to see that (\ref{coerce}) holds when $ \Omega$, $N$ and $D$ is a
standard domain for the mixed problem.

If $f_N$ is a function on $N$, then we may identify $f_N$ with an
element of the space $W^{-1/2,2}_D( \partial \Omega)$ by
$$
\langle f_N , \phi \rangle_{ \partial \Omega} =  \int_{N } f_N \phi \,
d\sigma, \qquad \mbox{for all } \phi \in W^ { 1/2,2} _D( \partial \Omega).  
$$
From Sobolev embedding we have $ W^ { 1/2, 2}_D( \partial \Omega)
\subset L^p ( \partial \Omega)$, where $p = 2 ( n-1)/(n-2)$ if $n\geq
3$ or $ p < \infty$ when $n=2$. Thus the integral on the right-hand
side will be well-defined if we have $f_N$ in $L^ { 2(n-1)/n}(
N )$ when $n \geq 3 $ or $ L^ p ( N )$ for
any $ p > 1$ when $n =2$.

We define a sub-linear operator $P$ which takes functions on $\partial \Omega$
to functions in $ \Omega$ by 
$$
Pf(x) = \sup _{ s > 0 } \frac 1 {s^ { n-1} } \int _{ \sball x s  } |f
|\, d\sigma, \qquad x \in \Omega 
$$
and a local version of  $P$ by 
$$
P_r f(x) = \sup _{ r> s > 0 } \frac 1 {s^ { n-1} } \int _{ \sball x s
} |f |\, d\sigma , \qquad \qquad x \in \Omega  . 
$$
On the boundary, we have that $Pf$ is the
Hardy-Littlewood maximal function 
$$
Mf(x)= Pf(x) = \sup _{ s> 0}   \frac 1 {s ^ {n-1}} \int _ { \sball x s} |f |\,
d\sigma, \qquad x\in \partial \Omega . 
$$
The following result is probably well-known, but we could not find a
reference.

\begin{lemma} 
\label{PEstimate}
For  $1< p < \infty$, $ 1 \leq q \leq pn/( n -1) $, $ x \in \partial \Omega
$ and  $r < r_0$, we have 
\begin{equation}\label{Plocal} 
\left ( \average _{\dball x r  }  |P_rf|^ q \, dy \right ) ^ { 1/q } 
\leq C \left  (\frac 1 { r^ { n-1} } \int _{ \sball x { 2r} } |f |^ p \, d\sigma \right
) ^ { 1/p} . 
\end{equation} 
The constant in this  estimate depends only on the Lipschitz
constant $M$ and the dimension.  
\end{lemma}

\begin{proof} We begin by considering the case 
  where $ \Omega = \{ (y', y_n ) :  y _
  n > 0\}$ is a half-space.
We use coordinates $ y = (y',y _n )$ and  we claim that 
\begin{eqnarray} 
\label{P1} 
Pf(y', y_n  )  & \leq  & Mf(y',0) 
\\
\label{P2}  
Pf(y) & \leq & C\|f\|_{ L^ p ( \partial \Omega) }  y _n ^ {( 1-n ) /p}, \qquad y _ n > 0 .
\end{eqnarray} 
The  estimate (\ref{P1}) follows easily since $ \sball {(y', y_n )} s \subset
\sball {(y', 0)} s $. 
To establish the second estimate, 
we observe that if $ s < y _n $, then $\sball y s = \emptyset$ and
hence 
$$
Pf(y)  =  \sup _{ s \geq  y _n } \frac 1 { s^ { n-1}} \int _ { \sball y s
} |f| \, d\sigma 
\leq C y _n ^ { (1-n)/p} \| f\|_{ L ^ p ( \partial \Omega) } .
$$

We claim that we have the following weak-type estimate for $Pf$, 
\begin{equation}
\label{Pclaim}
|\{ x\in \Omega :  Pf(x)  > \lambda \}|  \leq C \| f\|_ { L^ p (\partial  \Omega ) } ^ p
\lambda ^ { -pn /(n-1) } , \qquad \lambda > 0 . 
\end{equation}
To prove (\ref{Pclaim}), we may assume $ \| f \| _{ L^ p ( \partial
  \Omega )} = 1$. With this normalization, the observation (\ref{P2}) implies
that $\{  y ' : Pf(y', y _n) > \lambda \} = \emptyset $ if $ y _n > c
\lambda ^ { -p /( n-1)}  $. Thus, we may use Fubini's theorem to write
\begin{eqnarray*} 
|\{ x \in \Omega  :  Pf(x)  > \lambda \} | &  = & \int _ 0 ^ { c \lambda ^ { -p/( n-1)}  }
\sigma ( \{ y ' :Pf(y', y _n ) > \lambda \}) \, dy _n  \\
&\leq  & 
C\int _0 
^ {c\lambda ^ { -p/ ( n-1 ) }}   \sigma ( \{ y ' : Mf ( y',0) >
c   \lambda  \}) \, dy _n \\
&  = & C \lambda ^ { -p n / ( n-1)}   
\end{eqnarray*}
where we used (\ref{P1}), the weak-type $(p,p)$ inequality for the
maximal operator on $\reals ^ { n-1}$ and our normalization of the
$L^p$-norm of $f$. 

From the weak-type estimate (\ref{Pclaim}) and the Marcinkiewicz
interpolation theorem we obtain  that there is a constant $C$ depending on
$p$ and $n$ so that for $p>1$, 
\begin{equation}\label{Pglobal} 
\| Pf\|_{ L^ {pn/(n-1)} ( \Omega)} \leq  C \| f \|_{L^p ( \reals ^ { n
    -1} )} .  
\end{equation}
To obtain the  estimate (\ref{Plocal}), we observe that if $ y \in \ball x r $
then $ \ball y r \subset \ball x {2r}$ and hence 
$$
P_r f( y) \leq P _r  ( \chi _ { \sball x { 2r} } f ) (y) ,\qquad  y \in
\ball x r .
$$ 
Thus in the case where $ \Omega$ is a half-space, the result
(\ref{Plocal}) follows from (\ref{Pglobal}) and H\"older's inequality.
 
Finally, to obtain the  local  result on a general Lipschitz domain, 
one may change variables so that the boundary is flat near $ \sball x
r$.  This introduces the dependence on the  constant $M$. 
\end{proof}

\note
{ 
In applying the change of variables, it is helpful to note that for a
bi-Lipschitz transformation $ \Phi$, we have 
$$
B_ { cr} ( \Phi(x) ) \subset \Phi( B_r ( x ))  \subset B_ { Cr} ( \Phi(x)
).
$$

Do we need to multiply the radius by  a constant in the statement?

Check changes to Lemma \ref{RHEstimate} now that we use
Lemma   \ref{MSIRHI}. 

Does Lemma \ref{MSIRHI} below  duplicate one  of the estimates used
  in the $H^1$ part of 
  the paper. 
} 

We recall several versions of the  Poincar\'e and Sobolev
inequalities.
\begin{lemma} \label{YAPI}
Let  $ \Omega $ be a convex  domain of diameter $d$. 
Suppose that $ S $ is a subset of $  \bar \Omega $  that satisfies: 
a)  for some $ r $ with $ 0 < r< d$ we have 
 $ \sigma ( S \cap \ball x r ) = r ^ { n-1} $
and b)  there is a constant $A$ so that  $ \sigma ( S \cap \ball x t ) \leq A t^ { n-1}$ for $ t >0$. 
Let $u$ be a function in $ W^ { 1,p } ( \Omega)$ and suppose that
$u$ vanishes on $S$. 
Then  for $  1<  p < n $, we have  a constant $C$ 
$$
\left ( \int _ \Omega |u|^ p \, dy \right ) ^ { 1/p } 
\leq \frac { C d^ n  } { |\Omega | ^ { 1/p'} } { r ^ {1 - n / p }
} \left ( \int _ { \Omega } |\nabla u | ^ p \, dy \right) ^ { 1/p } .
$$
The constant $C$  depends on $p$, the dimension $n$ and $A$. 
\end{lemma} 
\begin{proof}
It suffices to consider functions $u$ which are smooth in $ \bar
\Omega$  and vanish on
$S$. 
 We follow the proof of Corollary 8.2.2 in the book of
Adams and Hedberg \cite{MR1411441}, except that we  substitute
the Riesz capacity for the Bessel capacity in order to obtain a
scale-invariant estimate.  Following their arguments, we obtain that
if $u$ vanishes on $S$, then
\begin{equation} 
\label{Fractional}
|u(x) | \leq \frac { d^ n }{ |\Omega| } ( I_ 1 ( |\nabla u | ) (x) 
+ \| \nabla u \|_{ L^ p ( \Omega) } \| I _ 1 ( \mu ) \| _ { L^ { p ' }
  ( \Omega ) } ) .
\end{equation} 
Here $ I _ 1( f ) ( x) = \int _ { \Omega } f(y) | x-y |^ { 1-n } \,
dy$ is the first-order fractional integral and $\mu $ is any
non-negative  measure
on $ S$ normalized so that $\mu(S) =1$.
To estimate $ \| I _1( \mu)\|_ {L^ { p'} ( \Omega)} $ we use Theorem
4.5.4 of Adams and Hedberg \cite{MR1411441} which gives that 
$$
\int _ { \reals ^ n } ( I _ 1( \mu ) ) ^ { p ' } \, dy \leq
C \int_{ \reals ^ n }  \dot W ^ \mu _ { 1, p } \, d\mu 
$$
where $\dot W ^ \mu _ { 1, p } (x) $ is the Wolff potential of $\mu$
defined by 
$$
\dot W ^ \mu _ { 1,p}(x) = \int _ 0 ^ \infty ( \mu ( \ball x t ) t ^ { p
  -n } ) ^ { 1/ ( p-1)} \, dt/t.
$$
Our assumptions imply that  with $ \mu = r ^ { 1-n } \sigma$ denoting
normalized surface measure on $S$, we have  $ I _ 1 ( \mu) (x) \leq C
r^ { ( p-n ) /( p-1)}$ where $C$ depends only on $A$. Using this
estimate for the Wolff potential and Young's convolution inequality to
estimate $I_1(|\nabla u|)$, the 
Lemma follows from (\ref{Fractional}). 
\end{proof} 

The next  inequality is also taken from Adams and Hedberg 
\cite[Corollary 8.1.4]{MR1411441}. Let $1/q + 1/n < 1 $ and  assume that
$ \Omega$ is a convex domain of diameter $d$. We let $ \bar u =
\average _\Omega u \, dy$  and then we may find  a constant $C =
C_{q,n}$ depending only on $q$ and $n$ so that  
\begin{equation}\label{SoPo1} 
\int _ { \Omega  } |u -\bar u | ^ q  \, dy 
\leq C\frac { d^ n } { |\Omega |} \left ( \int _  { \Omega  }
|\nabla u |^ {nq/(n+q)}\, dy  \right) ^ { ( n+q)/n} . 
\end{equation} 

Finally, we suppose that $\Omega$ is a domain and $ \dball x r$ lies
in a coordinate cylinder  $Z$ so that $ \partial \Omega  \cap Z$ lies
in a hyperplane and let $\bar u = \average _{\dball x r}
u\,dy$. Provided $ \dball x r \subset Z$, we have  
\begin{equation}\label{SoPo2} 
\left( \int _{ \sball x {r} } | u -\bar u | ^ { q} \, d\sigma
\right) ^ { 1/q} \leq C \left( \int_{ \dball x { r}}|\nabla u |^ p \,
dy \right ) ^ { 1/p }. 
\end{equation} 
In the inequality (\ref{SoPo2}),  $p$ and $q$  are related by $ 1/q  =  1/p - ( 1-
1/p) / (n-1) $ and $ p > 1$.

\begin{lemma} \label{MSIRHI} 
Let $ \Omega $, $N$ and $D$ be a standard domain for the mixed
problem.  Suppose that (\ref{SurfProp}) holds, let $ x \in \Omega$ and
$ 0 < r < r_0$.  Let $u$ be a weak solution of the mixed problem for a
divergence form elliptic operator with zero Dirichlet data and Neumann
data $f_N$. We have the estimate
$$
\left ( \average _{ \dball x  { r} }  |\nabla u | ^ 2 \, dy \right ) ^
  { 1/2 }  \leq C \left [ 
\average _ {\dball x  {2r} } |\nabla u | \, dy 
+\left (  \frac 1 { r^ { n-1}} \int  _ {N \cap  \sball x {2r} }  |f_N|^ { p} \, d\sigma \right) ^ { 1/p}
\right  ] . 
$$ 
Here, $p=2$ if $n  =  2$ and $ p = 2( n-1)/(n-2)$ for $n\geq 3$.
The constant $C$ depends only on $M$ and the dimension $n$.  
\end{lemma} 

\begin{proof} 
Changing variables to flatten the boundary of a Lipschitz domain
preserves the class of elliptic operators with bounded measurable
coefficients, thus it suffices to consider the case where the ball
$\sball x r $ lies in a hyperplane.  We may rescale to set $ r = 1$.
We claim that we can find an exponent $a$ so that  for $s$ and $t$
which satisfy  $ 1/2\leq s < t \leq 1$, we have
\begin{eqnarray}
\lefteqn{ \left ( \int _ { \dball x s } |\nabla u |^ 2 \, dy \right )
  ^ { 1/2} 
} 
\nonumber \\
& \leq &  \frac C { ( t-s ) ^ a } \left ( \int _ { \dball x t } | \nabla u
| ^ q \, dy \right ) ^ { 1/ q}  + 
\left (  \int _ {N \cap \sball x 1 } | f_N|^ { p} \, d\sigma  \right)
^ { 1/p}
\label{pqClaim} 
\end{eqnarray}
where we may choose the exponents $ p = 2 ( n-1)/( n-2) $ and $q =
2n/(2n+2)$ if $ n \geq 3$ or $ p =2$ and  $ q = 4/3$ if $n  = 2$. 

We give the details when $ n \geq 3$.  In the argument that follows,
let $ \epsilon = (t-s)/2$ and choose $ \eta$ to be a cut-off function
which is one on $\ball x s$, supported in $ \ball x { s+ \epsilon}$
and satisfies $ |\nabla \eta | \leq C /\epsilon $. We let $ v = \eta ^
2 ( u -E) $ where $E$ is a constant. If we choose $E$ so that $ v \in
W^ { 1,2 } _ D ( \Omega) $, the weak formulation of the mixed problem
and H\"older's inequality gives for $ 1 < p < \infty $
\begin{eqnarray} 
\int _ \Omega | \nabla u |^ 2 \eta ^ 2 \, dy & \leq & C \left [ \int _
  \Omega | u - E |^ 2 |\nabla \eta | ^ 2 \, dy + \left ( \int _{ N\cap
    \sball x { s+ \epsilon }} | u - E| ^ { p'}\, d\sigma \right ) ^ {
    2/ p ' } \right .  \nonumber 
\\ 
& & \qquad \left.  + \left( \int
  _{N \cap \sball x { s+ \epsilon } } |f_N|^ p \, d\sigma \right ) ^ {
    2/ p }\right] .
\label{John}
\end{eqnarray}

We consider two cases: a) $ \ball x { s+ \epsilon } \cap D = \emptyset
$ and b) $\ball x { s+ \epsilon } \cap D \neq \emptyset$.  In case a)
we may chooose $ E = \bar u = \average _ { \dball x { s+ \epsilon } }
u \, dy $.  We use the Poincar\'e-Sobolev inequality (\ref{SoPo1}) and
the inequality (\ref{SoPo2}) to estimate the first two terms on the
right of (\ref{John}) and conclude that
\begin{eqnarray*} 
\lefteqn{
\int _ { \dball x s } |\nabla u |^ 2 \, dy  
}\\
& \leq &  C \left [ \frac 1 { ( t-s) ^ 2 } \left ( \int _ { \dball x { s+
      \epsilon } }  |\nabla u | ^ { \frac {2n} {n+2} } \, dy \right )
  ^ { \frac { n+2 }  n }   \right .
\\
& & \quad + \left . \left ( \int _ { \dball x { s+ \epsilon } }
  |\nabla u | ^ 
{ \frac { np} { np-n + 1 } } \, dy \right ) ^ { \frac { 2 ( np -n + 1)} {  pn }  }   + 
\left (  \int _ { N \cap \sball x  1 }  |f_N|^ { p} d\sigma  \right) ^
      {\frac 1 p }
  \right ]. 
\end{eqnarray*}
If $n  \geq 3$, we may choose $p = 2 ( n-1) / ( n-2) $ and then we
have that $ np/( np -n + 1) = 2n /(n+2)$ to obtain the claim.

We now turn to case b).  Since $ \ball x { s+ \epsilon}$ meets the set
$D$, we cannot subtract a constant from $u$ and remain in the space of
test functions, $ W^ { 1,2 } _D ( \Omega)$. Thus, we  let $E=0 $ in 
(\ref{John}). 
 We let $ \bar u $ be the average value of $u$ on $ \dball x
{s+ 2\epsilon} $ and obtain 
$$
\int _ { \dball x { s+ \epsilon } } u ^ 2  |\nabla \eta |^ 2  \, dy 
\leq \frac C {  \epsilon  ^ 2} \left [ \int _ { \dball x  { s+ 2\epsilon }  } |
  u - \bar u | ^ 2 \, dy 
+  \bar u ^ 2 \right ] .
$$ 
Since $ \ball x { s+ \epsilon } \cap D \neq \emptyset $, our
assumption (\ref{SurfProp}) on the set $D$  implies that we may find a
point $ \tilde x \in \Lambda $ so that $\ball { \tilde x } \epsilon
\subset \ball x t $ and so that $ \sigma ( \ball { \tilde x } \epsilon
\cap D ) \geq c \epsilon ^ { n -1} $. As $c$ depends on $M$ our final
constant may be taken to depend on $M$.  Using (\ref{SoPo1}) and the
Poincar\'e inequality of Lemma \ref{YAPI} we conclude that
\begin{eqnarray}
\int _ { \dball x { s+ \epsilon }} u^2 |\nabla \eta | ^ 2  \, dy 
&\leq & C \left [ \frac 1 { \epsilon  ^ 2 } \left ( \int _ { \dball x
    {s+2\epsilon }  } |
  \nabla u | ^ { 2n / ( n+2) } \, dy \right ) ^ { ( n + 2) /n }
  \right.  \nonumber \\
\label{Paul}
& & \left.  \qquad 
+ \frac 1  {\epsilon^{  2n/q}} \left ( \int _ { \dball x { s + 2\epsilon } } 
  |\nabla u | ^ q \, dy \right ) ^ { 2/ q} \right]  
\end{eqnarray}
for $ 1 < q < n $. 
A similar argument using (\ref{SoPo2}) and Lemma \ref{YAPI} gives us 
\begin{eqnarray} 
\nonumber
\left( \int _ { \sball x { s+ 2\epsilon } } | u |^ { p'} \, d\sigma \right
) ^ { 1/ p' } 
& \leq &  
\left ( \int _ { \sball x { s + 2\epsilon }}  | u -\bar u | ^ { p
    ' } \, d\sigma   \right ) ^ { 1/ p' } + | \bar u | 
 \\ 
\nonumber
& \leq  & 
C \left [ \left ( \int _ { \dball x { s+ 2\epsilon } } |\nabla u | ^
  { np / ( np - n + 1)  } \, dy \right ) ^ {  ( np -n + 1 ) / ( np) }  \right.
\\
\label{George}
& &\left. \quad +  \epsilon ^ { 1- n/q} \left ( \int _ { \dball x { s+ 2\epsilon }
} |\nabla u | ^ q \, dy \right ) ^ { 1/ q}  \right ]  
\end{eqnarray}
where the  use of Lemma  \ref{YAPI} requires that we have   $ 1<
q< n$.
We use  (\ref{Paul}) and (\ref{George}) in (\ref{John}) and 
choose 
$ q = 2n/( n+2) $  and $ p = 2 ( n -2 ) /( n-1)
$ if $ n \geq 3$. 
Once we recall that $ t -s = 2
\epsilon $, we obtain (\ref{pqClaim}).   

Finally, we may use the techniques  given in \cite[pp.~80-82]{MR1239172}
or \cite[pp.~1004--1005]{FS:1984}
to see  that the claim (\ref{pqClaim}) implies the
estimate
\begin{equation} 
\left ( \int _ { \dball x {1/2}  } |\nabla u |^ 2 \, dy \right ) ^ { 1/2} 
\leq  C \left[  \int _ { \dball x 1 } | \nabla u
|   \, dy    +
\left (  \int _ {N \cap  \sball x 1 }  |f_N|^ { p}  d\sigma \right) ^ { 1/p}
  \right]
\end{equation} 
with $p$ as in (\ref{pqClaim}).  

When the dimension $n=2$, the exponent $2n/(n+2)$ is 1 and it is not
clear that we have  (\ref{SoPo1}) as used  to
obtain (\ref{Paul}).   However, from (\ref{SoPo1}) and H\"older's
inequality we can show 
$$
\left( \int_{ \dball x { s+ 2\epsilon } } |u - \bar u | ^ 2 \, dy
\right ) ^ { 1/2} \leq C \left( \int_{ \dball x { s+ 2\epsilon } }
|\nabla u | ^ {4/3}  \, dy
\right ) ^ { 3/4} .
$$
This  may be substituted for (\ref{SoPo1})  in the above argument to obtain
(\ref{pqClaim}) when $ n =2$. 
\end{proof}

\begin{lemma} 
\label{RHEstimate}
Let $ \Omega$, $D$ and $N$ be a standard domain for the mixed
problem.   Let $ x \in \Omega$ and suppose that $r$ satisfies $ 0< r <
r_0$. 
Let $u$ be a weak solution of the mixed problem
(\ref{WeakMix}) with  zero  Dirichlet data  and Neumann data $f$ in
$L^p(N)$ which is supported in  $ N \cap \sball x r $. 
There exists  $ p_0=p_0(n,M)  > 2 $ so 
 that for $t $ in  
  $[2,p_0) $ if $n \geq 3$ or $t$ in $(2,p_0)$ if $n =2 $, we have
the estimate 
\begin{eqnarray*} 
\lefteqn{ \left (  \average _{ \dball x r  } |\nabla u |^t \, dy  \right
)^ { 1/t}}   \\
& \leq &  C\left[   \average _ {\dball  x {2r} } |\nabla u  |\,dy
  +\left( \frac 1 { r^ { n-1} } 
   \int  _{  \sball x {2r}  \cap N }  |f|^{t(n-1)/n}\,
d\sigma\right) ^ { n/(t(n-1))}\right] .
\end{eqnarray*}
The constant in this estimate depends on $t$,  $M$ and $n$. 
\end{lemma} 
\note{
In applications, we seem to only need $f$ bounded. Is it worth the
trouble to keep track of the exponents? 
}

\begin{proof} 
According to Lemma \ref{MSIRHI}, $ \nabla u$ satisfies a reverse
H\"older inequality and thus we may apply a result of Giaquinta
\cite[p.~122]{MG:1983} to conclude that there exists $p_0 > 2$ so that we
have
$$
\left( \average _ {\dball x r  } | \nabla u|^  t
\, dy \right) ^ { 1/t} \leq C\left [ \average_ {\dball  x {2r} } |\nabla u | \, dy 
+ \left(  \average _ { \dball  x { 2r} }
(P_ {2r} |f|^p)^{ t/p}  \, dy \right) ^ { 1/ t }\right]
$$
for $t $ in $[2,p_0)$ and $p$ as in Lemma \ref{MSIRHI}. 
From this, we may use Lemma \ref{PEstimate} 
to obtain 
$$
\left( \average _ {\dball x r  } | \nabla u|^  t
\, dy \right) ^ { 1/t} \leq C\left [ \average_ {\dball  x {2r} } |\nabla u | \, dy 
+ \left(  \average _ { \sball  x { 4r} }
 |f|^{ t(n-1)/n}  \, d\sigma \right) ^ { n/ t(n-1) }\right]
$$
when $ n \geq 3$ and $ t $ is in $[2,p_0)$. If $n=2$ we need $ t >2$
so that $f$ is raised to a power larger than 1. 
Now a  simple  argument that involves covering $ \sball x r$ by
surface balls of radius $r/4$ allows us to 
 conclude the estimate
of the Lemma.  
\end{proof} 

\section{Estimates for solutions with atomic data}
\label{Atoms}

We establish an estimate for the solution of the mixed problem when
the Neumann data is an atom for $ H^1$ and the Dirichlet data is zero.
The key step is to establish decay of the solution as we move away
from the support of the atom.  We will measure the decay by taking
$L^q$-norms in dyadic rings around the support of the atom. Thus,
given a surface ball $ \sball x r$, $ x \in \partial \Omega$, we
define $\Sigma_k = \sball x { 2^ k r} \setminus \sball x {2^ { k-1}
r}$ and define $ S_k = \dball x {2^k r } \setminus \dball x { 2^ {
    k-1} r } $.

\begin{theorem} \label{AtomicTheorem} 
Let $ \Omega$, $N$ and $D$ be a  standard  domain for the mixed problem. Let $u$
be a weak solution of the mixed problem with Neumann data $a$ which is
an atom which is supported in   $N \cap\sball x r $ and zero Dirichlet
data. 

If $p_0$ is as in Lemma \ref{RHEstimate} and $ 1< q < p_0/2$, then we
have $ \nabla u \in L^ q ( \partial \Omega)$,  
\begin{equation} \label{LocalPart}
\left( \int _{\sball x {8r} }   |\nabla u |^q \, d \sigma  \right)^ {
  1/q} 
\leq C \sigma (\sball x  {8r} )^ {-1/q'}
\end{equation}
and for $ k \geq 4$, 
\begin{equation} \label{Decay}
\left ( \int _{ \sring k}  |\nabla u |^q \,d\sigma  \right) ^ { 1/q} 
\leq 
C 2^ {-\beta k}  \sigma(  \sring k ) ^ {- 1/q'} . 
\end{equation}
Here, $ \beta $ is as in Lemma \ref{Green}  and the constant $C$ in
the estimates (\ref{LocalPart}) and (\ref{Decay})
depends on $q$ and  the global character of the domain. 
\end{theorem}

If $r < r_0 $ and    $x$ is in $\partial \Omega$, then we may
construct a star-shaped Lipschitz domain $ \locdom x r  =
Z_r(x) \cap \Omega$  where $\cyl x r $ is the coordinate cylinder defined
above. Given a function $v$ defined in $ \Omega$, $x \in \partial
\Omega$,  and $r>0$, we
define a truncated non-tangential maximal function $ \nontan{v_r} $ by 
$$
\nontan v_r (x) = \sup _{ y \in \ntar  x  \cap \ball x r } |v(y)|.
$$

\begin{lemma} 
\label{NeumannRegularity} Let $\Omega$ be a Lipschitz domain. 
Suppose that $ x \in \partial \Omega $ and $0< r < r_0$.  Let $u$ be a
harmonic function in $ \locdom x {4r}   $. 
If $\nabla u  \in L^2 ( \locdom x {4r} )$ and
 $ \partial u /\partial \nu  $ is in $L^ 2 (\partial \Omega \cap
\partial \locdom x {4r} ) $, then we  have  $ \nabla u  \in L^2 ( \sball x r)$ and 
$$
\int _ { \sball x {r}}  (\nontan {( \nabla u )}_r)^2 \, d\sigma
\leq C \left ( \int _ { \partial \Omega \cap \partial \locdom x {4r}  } \left |\frac { \partial u  }{ \partial
    \nu } \right | ^ 2 \, d\sigma 
+ \frac 1 r  \int _ {\locdom x {4r}  } |\nabla u |^2 \,
dy\right).
$$
The constant $C$ depends only on the dimension $n$ and $M$. 
\end{lemma}

\begin{proof} 
Since the estimate only involves $\nabla u$,  we may
subtract a constant from $u$ so that $ \int _ {\locdom x r } u \, dy =
0$. 
We pick a smooth cut-off function $ \eta$ which is one on $ 
Z_{3r}(x)$ and zero outside $Z_{4r}(x)$.  Since we assume  that
$\nabla u$ is in $L^2( \locdom x {4r } )$,  it follows that  $\Delta ( \eta u) =
u \Delta  \eta  + 2 \nabla u \cdot \nabla \eta $ is in
$L^2( \locdom x {4r} )$. Thus, with $ \funsol$ the usual fundamental solution
of the Laplacian, $w = \funsol*(\Delta( \eta u))$ will be in the
Sobolev space $W^ {2,2}( \reals ^n)$. We have defined $ \Delta (\eta u
)$ to be zero outside $ \locdom x {4r}$ in order to make sense of the
convolution in the definition of  $w$.     Next, we 
let $v$ be the solution of  the  Neumann problem 
$$
\left\{ \begin{array} {ll} 
\Delta v = 0, \qquad  & \mbox {in } \locdom x {4r
}  \\
\bigfrac { \partial v } { \partial \nu } = \bigfrac { \partial( \eta  u) }
{ \partial \nu }- \bigfrac {\partial w }{ \partial \nu},  \qquad  & \mbox{on
}\partial  \locdom x  {4r  } . 
\end{array}
\right.
$$
According to Jerison and Kenig \cite{JK:1982c}, the solution $v$ will
have non-tangential maximal function in $L^2 ( \partial \locdom x {4r}
)$.  By
uniqueness of weak solutions to the Neumann problem, we may add a
constant to $v$ so that we have $ \eta u = v+w$.  As $w$ and all its
derivatives are bounded in $ \locdom x {2r}$ and the non-tangential
maximal function of $\nabla v$ is in $L^2( \bdry \locdom x {4r})$, we
obtain the  Lemma. 
\end{proof}

The proof
of the following   Lemma for the regularity  problem  is identical to
the proof of Lemma \ref{NeumannRegularity}.

\begin{lemma} 
\label{DirichletRegularity}
Let $\Omega$ be a Lipschitz domain. 
Suppose that $ x \in \partial \Omega $ and $0<  r < r_0$.  Let $u$ be a
harmonic function in $  \locdom x { 4r} $. 
 If $\nabla u  \in L^2 (\locdom x {4r} )$ and
 $ \tangrad  u $ is in $L^ 2 (\partial \Omega \cap \partial \locdom x { 4r} )$, then we
 have  $ \nabla u  \in L^2 ( \sball x r)$ and 
$$
\int _ { \sball x {r}}  (\nontan {( \nabla u )}_r)^2 \, d\sigma
\leq C \left( \int _ { \partial \Omega \cap \partial  \locdom x { 4r} }| \tangrad u |^2 \, d\sigma 
+  \frac 1 r \int _ {\locdom x { 4r} }  |\nabla u |^2 \,
dy\right).  
$$
The constant $C$ depends only on the dimension $n$ and $M$. 
\end{lemma} 

The following weighted estimate will be an intermediate step towards
our estimates for solutions with atomic data.
In the next lemma,   $\Omega$  is  a bounded Lipschitz domain and 
the   boundary is written $ \partial \Omega = D \cup
N$. Recall that $ \delta(x)$ denotes the distance from
$x$ to the set $\crease$.

\begin{lemma} \label{Whitney}
Let $\Omega$, $D$ and $N$ be a standard domain for the mixed problem. 

Let $u$ be a weak solution of the mixed problem (\ref{WeakMix}) with
Neumann data $f_N \in L^2 (N)$ and zero Dirichlet data. 

Let $\epsilon \in \reals$, $x \in \partial \Omega$ and $0< r < r_0$
and assume that for some $A>0$,  $ \delta (x) \leq Ar$. Then we have 
$$
\int _ {\sball x r    }   (\nontan{( \nabla u)}_{c\delta}) ^2 \, \delta  ^{1-\epsilon}  d\sigma 
\leq C \left ( \int _ { \sball x {2r}}
|f_N|^2  \delta  ^ { 1- \epsilon } \,  d\sigma
  + \int _ { \dball x { 2r}   }  |\nabla u  |^2 \, \delta  ^ {
    -\epsilon } \, dy
\right )  . 
$$
The constant in this estimate depends on $M$, $n$,  $\epsilon$  and $A$. 
\end{lemma}

\note
{
Using a Hardy inequality, we can probably show that $u$ in $L^2( N
;   \delta \, d\sigma )$ implies that $u$ is in the dual of $W^ { 1/2,
    2}_D ( \partial \Omega)$.  

}
The proof below  uses a Whitney decomposition and thus it is simpler if
we use surface cubes, rather than the  surface balls used elsewhere. 
A {\em  surface cube }is the image of a 
cube in $ \reals ^ { n-1}$ under the  mapping  $
 x' \rightarrow ( x' , \phi(x'))$. Obviously, each cube will lie in a
 coordinate cylinder.

\begin{proof}
We may assume  that $ \dball x {2r}$ is contained in a coordinate
cylinder $ Z_{ 2r_0}$. If $Z_ { 100r_0} \cap \partial \Omega \subset N$
or $Z_ { 100r_0} \cap \partial \Omega \subset D$, then the estimate of
the Lemma follows easily from   Lemma \ref{NeumannRegularity} or Lemma
\ref{DirichletRegularity} since we have that $ \delta (y) $ is
equivalent to $r$ for $y \in \dball x {2r}$.  This equivalency follows
from our assumption that $ \delta (x) < Ar$ and that $Z_{ 100r_0}$
does not intersect $\Lambda$. 

If  $ Z_ { 100r_0}$ meets both $D$ and $N$, we begin by finding a
decomposition of $ ( \partial \Omega \cap  Z_{ 4r_0})  \setminus \Lambda $ into non-overlapping
surface cubes $ \{ Q_j \}$ which satisfy:  
1) For each cube $ Q_j$, we have constants $c''$ and $c'$ so that
$c''\delta (y) \leq \diam (Q_j) \leq c ' \delta (y)$ for $ y \in
Q_j$. The constant 
$c'$ may be chosen as small as we like. 
2) We let $T(Q) = \{ y \in \Omega : \dist (y, Q ) < \diam Q\}$. Then
the family $ \{ T(2Q_j)\}$ has bounded overlaps and thus 
$$
\sum \chi _{ T(2Q_j)}\leq C (n, M, c'').
$$
To construct the family of surface cubes, begin with a Whitney
decomposition of $ \reals ^ { n-1}\setminus \{ ( \psi (x''), x'' ):
x''\in \reals ^ { n -2}\}$ and then map the cubes onto the boundary
with the map $x' \rightarrow ( x', \phi(x'))$. Here, $ \phi$ and $\psi
$ are the functions used to describe $\partial \Omega$ and $\Lambda$
in the coordinate cylinder $Z_{ r_0}$.

As the surface cubes $ Q_j$ are connected and $ \delta $ never vanishes on
$Q_j$, we have that either $ Q_j \subset N$ or that $ Q_j \subset
D$. We choose the constant $ c'$ small so that $ Q _ j \cap \sball x r
\neq \emptyset$ implies that $T(2Q_j)\subset \dball x { 2r}
$. Let $r_j$ be the diameter of the cube $r_j$. Applying Lemma \ref{NeumannRegularity} or Lemma
\ref{DirichletRegularity}, we conclude that 
\begin{equation} \label{Whit1}
\int _ { Q_j } |\nabla u |^ 2 \, d\sigma \leq C \left ( \int _ { 2Q_j
  \cap N} \left| \frac { \partial u } { \partial \nu } \right | ^ 2 \,
d\sigma + \frac 1 { r_j } \int _ { T( 2Q_j)} |\nabla u |^ 2 \, dy
\right) . 
\end{equation}
We multiply equation (\ref{Whit1}) by $ r_j ^ { 1-\epsilon } $, choose
$c'$ small so 
that $ r _j $ is equivalent to $ \delta (y)$ in $T(2Q_j)$ and obtain 
\begin{equation} \label{Whit2} 
\int _ { Q_j } |\nabla u |^ 2  \delta ^ { 1- \epsilon}  \, d\sigma \leq C \left ( \int _ { 2Q_j
  \cap N} \left | \frac { \partial u } { \partial \nu } \right | ^ 2
\delta ^ { 1- \epsilon } \,
d\sigma +  \int _ { T( 2Q_j)} |\nabla u |^ 2  \delta ^ { -\epsilon }\,
dy 
\right) . 
\end{equation}
We sum over $j$ such that $ Q_j \cap \sball x r \neq \emptyset$ and
use that the family $ \{ T(2Q_j)\}$ has bounded overlaps to obtain the
Lemma. 
\end{proof} 

An important part of the proof of our estimate for the mixed problem
is to show that a solution with Neumann data an atom will decay as we
move away from the support of the atom.  This decay is encoded in
estimates for the Green function for the mixed problem. These
estimates rely in large part on the work of de Giorgi \cite{EG:1957},
Moser \cite{JM:1961} and Nash \cite{JN:1958}, on H\"older continuity
of weak solutions of elliptic equations with bounded and measurable
coefficients, and the work of Littman, Stampacchia and Weinberger
\cite{LSW:1963} who constructed the fundamental solution of such
operators.  Also, see Kenig and Ni \cite{MR87f:35065} for the
construction of a global fundamental solution in two dimensions.
Given the free space fundamental solution, the Green function may be
constructed by reflection in a manner similar to the construction
given for graph domains in \cite{LCB:2008}.  A similar argument was
used by Dahlberg and Kenig \cite{DK:1987}  and by Kenig and Pipher
\cite{KP:1993} in their studies of the
Neumann problem.   Once we have a Green function which satisfies the
correct boundary conditions in a coordinate cylinder, we may solve a
weak version of the mixed problem to obtain a Green function in all of
$ \Omega$.

\begin{lemma} \label{Green}
Let $\Omega$, $N$ and $D$ be a standard domain for the mixed problem.
There exists a Green function $G(x,y)$ for the mixed problem which
satisfies: 
1) If $G_x(y) = G(x,y)$, then $ G_x$ is in $W^ { 1,2}_D (
\Omega \setminus \ball x r )$ for all $r>0$, 
2) $\Delta G_x = \delta _x$, the Dirac $\delta$-measure at $x$, 
3) If $f_N$ lies in $ W^ {-1/2, 2} _D ( \partial \Omega)$, then the
solution of the mixed problem with $f_D=0$ can be represented by
$$
u ( x) =  - \langle f_N , G_x\rangle _{\partial \Omega} , 
$$
4) The Green  function  is H\"older continuous away from the pole and
satisfies the estimates 
$$
|G(x,y) - G(x,y')| \leq \frac { C|y-y'|^ \beta } { |x-y |^ { n-2+\beta
  }} , \qquad |x-y| > 2 |y-y'|, 
$$
$$| G(x,y) | \leq \frac C { |x-y|^ { n-2} }, \qquad n\geq 3,  $$
and with $ d = \diam( \Omega)$, 
$$| G(x,y) | \leq C( 1+ \log (d/ |x-y|)) , \qquad n =  2.  $$
\end{lemma} 

\begin{lemma}  \label{Energy}
Let $u$ be a weak solution of the mixed problem  (\ref{WeakMix}) with Neumann data $f$
in $L^ {p}(N)$  where $p = (2n-2)/n  $ for $ n\geq 3$.  Then we have the
estimate 
$$
\int _{\Omega } |\nabla u | ^2 \, dy \leq  C\| f\|^2_ {L^ p ( N)} .
$$

If $n =2$, we have 
$$
\int _{\Omega } |\nabla u | ^2 \, dy \leq  C\| f\|^2_ {H^1 (N)} .
$$

In each case, the constant $C$ depends on $\Omega$ and the constant in
(\ref{coerce}).  
\end{lemma}

\begin{proof} When $ n \geq 3$, we use that $ W^ { 1/2, 2}_D( \partial
  \Omega ) \subset L^ { 2(n-1) /( n-2)} ( \partial
  \Omega)$.  By duality, we see that $ L^ {  2(n-1)/n}(
\partial \Omega) \subset W^ { -1/2, 2}_D( \partial  \Omega ) $ and
since the  weak  solution of the mixed problem
satisfies 
$$
\int_\Omega |\nabla u |^2 \, dy \leq C  \| f\|^2 _ { W^ { -1/2,2}_D( \partial
  \Omega)}
$$
the Lemma follows. 

When $ n=2$, the proof above fails since  we do not have  $
W^ { 1/2, 2}_D(\partial \Omega ) \subset L^ \infty( \partial
\Omega)$. However, we do have the embedding 
$W^ { 1/2, 2}_D(\partial \Omega ) \subset BMO(\partial \Omega)
$. Since $ \phi \in W^ { 1/2, 2}_D( \partial \Omega)$ vanishes on $D$
and $ f \in H^1(N)$ has an extension $\tilde f$ which lies in $H^1(
\partial \Omega)$, we obtain the result for $n=2$. 
\end{proof}

Finally, we give  a technical lemma that will be used below.
\begin{lemma}
\label{DorN}
 Let $ \Omega $, $N$ and $D$  be a standard domain for
  the mixed problem and suppose that $0 < r < r_0$,  $ x \in \partial
  \Omega$
and  $   \delta (x) > r \sqrt { 1+M^2}$. Then we have $ \sball x r
\subset N$ or $ \sball x r \subset D$. 
\end{lemma}

\begin{proof} 
We fix $ y \in \sball x r $. Since $ r< r_0$, we may find a coordinate
cylinder $Z$ which contains $ \sball x r$. We let $ \phi $ be the
function whose graph gives $ \partial \Omega$ near $ Z $. Since $ y
\in\sball x r$, we have $|x'-y'| < r$. We let $ x'(t) = ( 1-t) x' + t
y '$ and then $ \gamma (t) = ( x' (t) , \phi (x' (t)))$ gives a path
in $ \partial \Omega$ joining $x$ to $y$ and of length at most $
r\sqrt { 1+ M^ 2}$. Since $ \delta (x )> r\sqrt { 1+M^2}$ and $
\delta$ is Lipschitz with constant one, we have that $ \delta(
\gamma(t) ) >0$ for $ 0 \leq t \leq 1$. Since $ \gamma(t)$ does not
pass through $ \Lambda $ we must have $x$ and $y$ both lie in $D$ or
both lie in $N$. As $y$ is an arbitrary point in $\sball x r$, it
follows that $ \sball x r $ lies entirely in $D$ or entirely in $N$.
\end{proof}

\begin{proof}[Proof of \ref{AtomicTheorem}]
It suffices to restrict attention to atoms which are supported  in a surface
ball $\sball x r$, with $ x\in\partial \Omega$ and $0<  r< r_0$ since
an atom which  is supported in  a larger surface
ball can be sub-divided into a finite number of atoms which are
supported  in
balls of the form $ \sball x { r_0}$. The increase in the constant due
to this step  will depend on the global character of the domain. 

Thus, we fix an atom $a$ that is  supported in  the set $ \sball x r
\cap N$ and begin the proof of (\ref{LocalPart}). We consider two cases: 
a) $ \delta (x) \leq  16r \sqrt { 1+M^2}$, and b) $ \delta (x) > 16r
\sqrt { 1+M^2}$. 

In case a) we fix $ q$ between 1 and 2 and use H\"older's inequality
with exponents $ 2/q$ and $2 /(2-q)$ to find 
\begin{eqnarray*}
\left( \int _{ \sball x { 8r}} |\nabla u | ^ q\, d\sigma \right) ^ {
  1/q} 
& \leq&  \left( \int_ { \sball x { 8r}} |\nabla u |^ 2\delta^ { 1-\epsilon} \, d\sigma \right) ^
     { 1/2} \left ( \int _ { \sball x { 8r}} \delta  ^ { \frac { q (
       \epsilon -1)} { 2-q}}\,d \sigma \right) ^ { \frac {2-q}{ 2q}} \\
& \leq &  C r^ { (n-1) ( \frac 1 q -\frac 1 2 ) + \frac { \epsilon -1}2}
\left( \int_{ \sball x { 8r}}|\nabla u | ^2 \delta ^ { 1- \epsilon} \,
d\sigma \right)^ { 1/2} .
\end{eqnarray*}
The second inequality requires that $q$ and $ \epsilon$ satisfy $ q (
\epsilon -1) /( 2- q) > -1$ or $ q < 1 /( 1-\epsilon /2)$. Next, we
use Lemma \ref{Whitney} and our assumption that $ \delta(x) \leq 16r
\sqrt { 1+M^2}$  to bound the weighted $ L^ 2 ( \delta ^ {
  1-\epsilon } d\sigma )$ norm of $ \nabla u $. This gives us 
\begin{eqnarray*}
\left( \int _ { \sball x {8r} } |\nabla u | ^ q \, d\sigma \right) ^ {
  1/q} &  \leq &  C \left [ \left ( \int _ { \sball x { r} \cap N  }
  |a |^ 2  \delta ^ { 1-\epsilon} \,    d\sigma \right ) ^ { 1/2} \right. 
\\
& & \quad  + \left . \left ( \int _  { \dball x { 16r}} |
  \nabla u | ^ 2 \delta ^ { -\epsilon } \, dy \right ) ^ { 1/2} \right
] r ^ { ( n-1)( \frac 1 q - \frac 1 2  ) + \frac { \epsilon -1 } 2} .
\end{eqnarray*}
We estimate the integral over $ \dball x {16r}$ in this  last
expression with H\"older's inequality and obtain 
\begin{eqnarray*}
\lefteqn{ \left ( \int _  { \dball x { 16r}} |
  \nabla u | ^ 2 \delta ^ { -\epsilon } \, dy \right ) ^ {1/ 2 }
} \qquad
\\
& \leq  & 
C \left ( \int _ { \dball x {16r}}  |\nabla u | ^ p \, dy \right) ^ {
    1/p} \left ( \int _ { \dball x { 16r} } \delta ^ { - \epsilon p/(
    p-2)}\, dy \right ) ^ {  1/ 2  -  1/ p }   \\
& \leq &  C r ^ { n ( \frac 1 2 - \frac 1 p) - \epsilon / 2 }\left ( \int _
     { \dball x { 16r}}|\nabla u | ^ p \, dy \right ) ^ { 1/p} .
\end{eqnarray*}
The second inequality depends on our assumption on $ \Lambda $ and
holds when $ \epsilon p /( p-2) < 2 $ or $p > 2 / ( 1- \epsilon /2)$. 
Now we may use the three previous displayed equations and Lemma
\ref{RHEstimate} to obtain 
$$
\left ( \frac 1 { r^ { n-1}}\int _ { \sball x { 8r}}|\nabla u | ^ q \,
d\sigma \right ) ^ { 1/q} \leq C \left [ \left ( \frac 1 { r^ n } \int
  _ {\dball x { 32r}} |\nabla u |^ 2\, dy   \right ) ^ { 1/2} + r
  ^ { 1-n } \right ].
$$
In this last step, we have used the normalization of $a$, $ \| a \|_{
  L^ \infty } \leq 1/ \sigma ( \sball x r )$  to estimate
the term involving the Neumann data from Lemma
\ref{RHEstimate}. Finally, we may use the Lemma \ref{Energy} and the
normalization of the atom to obtain that $( r^ { -n} \int _ \Omega
|\nabla u |^ 2 \, dy  ) ^ { 1/2} \leq C r^ { 1-n}$ which gives
the estimate (\ref{LocalPart}). 

In case b), we use Lemma \ref{DorN} to conclude that  $ \sball x { 16r} \subset
N$. Next, 
we use H\"older's inequality, that $a$ is supported in $ \sball x r$ and Lemma
\ref{NeumannRegularity} to obtain 
\begin{eqnarray}
\lefteqn{ \left(  \frac 1 { r ^ { n-1}} \int _{ \sball x { 8r} } |\nabla u | ^ q \, d\sigma \right )
^ { 1/q} } \nonumber \\
& \leq  &C \left (\frac 1 { r ^ { n-1}} \int  _ { \sball x { 8r} } |\nabla u | ^ 2
\, dy \right ) ^ { 1/2}  \nonumber  \\
& \leq   & 
C \left [ \left ( \frac 1 { r^ { n-1}}  \int  _ { \sball x r \cap N }
  |a|^2\, d\sigma \right ) ^ { 1/2}
   + \left(\frac 1 { r ^ n } \int _ {
    \dball x { 16r}}|\nabla u | ^ 2 \, dy \right ) ^ { 1/2} \right]. 
\label{EasyCase}
\end{eqnarray}
Using the normalization of the atom $a$ and Lemma \ref{Energy},  the
right-hand side of (\ref{EasyCase}) may be estimated by $
\sigma(\sball x { 8r} ) ^ { -1} $ and we obtain (\ref{LocalPart}) in
this case. 

Now we turn our attention to the proof of the estimate 
(\ref{Decay}). Our first step is to observe that the solution $u$
satisfies the estimate
\begin{equation}
\label{AtomDecay}
|u(y) | \leq \frac { Cr ^ \beta  } { |x-y |^ { n-2+ \beta  } }, \qquad
|y-x| > 2r.
\end{equation}
To establish (\ref{AtomDecay}), we begin with the representation
formula in part 3) of Lemma \ref{Green} and claim that we may find
$\bar x $ in $\sball x r $ so that 
$$
u ( y ) = - \int _ { \sball x r \cap N } a(z) ( G(y,z) - G( y, \bar x)
) \, d\sigma 
.
$$
If $ \sball x r \subset N$, then we may let $ \bar x = x$ and use that
$a$ has mean value zero to obtain the above representation. If $\sball x r \cap D
\neq \emptyset $, then we choose $\bar x \in D \cap \sball x r$ and use
that $  G( y, \cdot ) $ vanishes on $D$. Now the estimate
(\ref{AtomDecay}) follows easily from the normalization of the atom
and the estimates for the Green function in part 4) of Lemma
\ref{Green}.  

We will consider three cases in the proof of (\ref{Decay}): 
a) $2^k r < r_0$ and $ \delta (x) \leq 2\cdot 2^k r\sqrt{ 1+M^2}$, b) 
$2^k r < r_0$ and $ \delta (x) >  2\cdot 2^k r\sqrt { 1+M^2}$, c) $2^k r
 \geq r_0$. The details are similar to the proof of (\ref{LocalPart}),
 thus we will be brief. 

We begin with case a) and use H\"older's inequality with
exponents $2/q$ and $ 2/(2-q)$ to obtain 
$$
\left( \int _ { \Sigma _k } |\nabla u | ^ q \, d\sigma \right ) ^ {
  1/q} 
\leq C \left( \int _ { \Sigma _k } |\nabla u |^2 \delta ^ { 1-
  \epsilon } \, d\sigma \right ) ^ {1/2} ( 2^ k r ) ^ { (n-1) ( \frac
  1 q - \frac 1 2 ) + \frac { \epsilon -1} 2 }.
$$
As in the proof of the estimate (\ref{LocalPart}),  this requires that
$ 1 < q <  1/ ( 1 - \epsilon /2)$.  
From Lemma \ref{Whitney} we have 
$$
\left( \int _ { \Sigma_k } | \nabla u | ^ 2 \delta ^ { 1-\epsilon } \, d\sigma \right) ^ { 1/2} 
\leq C \left ( \sum _ { j = k -1} ^ { k+1} \int _ { S_j } |\nabla u |
^ 2\delta ^ { -\epsilon } \, dy  \right) ^ { 1/2} 
$$
This estimate requires $ k \geq 2$ so that $ \sring {k-1} \cap \sball
x {r} = \emptyset$. 
Then H\"older and the reverse H\"older estimate in Lemma
\ref{RHEstimate} gives 
$$
\left( \int _ {S_k} |\nabla u | ^ 2 \delta ^ { -\epsilon } \, dy
\right) ^ {\frac 1 2 } \leq C \left ( \int _ { S_k } |\nabla u | ^ p \, dy
\right) ^ { \frac 1 p  } 
\leq C ( 2^ k r) ^ {- \frac {  \epsilon}  2} \left ( \sum _ { j = k -1} ^ {
  k+1} \int _{S_j} |\nabla u |^2 \, dy \right) ^ { \frac 1 2 } .
$$
Here we need $ k \geq 2$ so that $ \sball x r \cap \Sigma _ { k-1} =
\emptyset $ and the term in  involving the Neumann
data in Lemma \ref{RHEstimate} vanishes. Finally, from Caccioppoli and our estimate 
(\ref{AtomDecay}) for $u$, we obtain that 
$$
\left( \int  _{S_k} |\nabla u | ^ 2 \, dy \right) ^ { 1/2} 
\leq \frac C { 2^ k r} \left( \sum _ { j = k -1} ^ {k+1} \int_{S_j}
|u|^2 \, dy \right) ^ { 1/2} 
\leq C 2 ^ { -k\beta} (2^k r ) ^ {1 -n/2}. 
$$
Again, we need $ k\geq 2$ so that the data for the mixed problem is
zero when we apply Caccioppoli's inequality. 
Combining the four previous estimates gives 
$$
\left( \int _ { \Sigma _k } |\nabla u | ^ q \, d\sigma \right) ^ {
  1/q} \leq C 2 ^ { -k\beta} \sigma ( \Sigma _k ) ^ { -1/q'}
$$
for $k \geq 4 $. We need $k \geq 4$ in order to fatten up the set $ \sring k $
three times: once to apply Lemma \ref{Whitney}, once to apply Lemma
\ref{RHEstimate} and once to apply Caccioppoli's inequality. 

Now we consider case b). Since $ \delta (x) \geq 2( 2^ k r) \sqrt {
  1+M^ 2} $, we have $ \sball x {2 \cdot 2^ k r  } \subset N$  by 
Lemma \ref{DorN}. 
Hence, we may use Lemma \ref{NeumannRegularity},
Caccioppoli's inequality and (\ref{AtomDecay}) to obtain
(\ref{Decay}). 

Finally, we consider case c) where $ 2^k r > r_0$. We recall that we
have a covering of $ \partial \Omega$ by coordinate cylinders.  In each coordinate
cylinder, we may use Lemma   \ref{NeumannRegularity}, Lemma
\ref{DirichletRegularity}
or Lemma \ref{Whitney} and the techniques
given above to obtain
$$
\left( \int _ { Z_ { r_0} \cap \partial \Omega } |\nabla u |^ q \, d\sigma\right)^ { 1/q} 
\leq C r_0 ^ { (1-n)/q'}.
$$
Adding these estimates gives (\ref{Decay}) with a constant that depends
on the global character of the domain. 
\end{proof}

We now show that the non-tangential maximal function of our weak
solutions  lies in $L^1$ when the Neumann data is an atom.

\begin{theorem} \label{HardyTheorem}
Let $ \Omega$, $N$ and $D$ be a standard domain for the mixed
problem. 
If  $ f_N$ is in $ H^1 (N)$, then there exists $u$ a solution of the $L^1$-mixed
problem (\ref{MP}) with Neumann data $f_N$ and zero Dirichlet data  and this
solution satisfies 
$$
\|\nontan{ (\nabla u)  } \|_{L^1( \partial \Omega)}  \leq C \|f_N\|_
{ H^1 (N)}. 
$$
The constant $C$ in this estimate depends on the global character of $
\Omega$, $N$ and $D$. 
\end{theorem} 

\note
{
Theorem restated to give existence in the Hardy space, rather than for
an atom.

Check  to make sure the proof proves the theorem. 
}

\begin{proof} 
We begin by considering the case when $f_N$ is an atom and we let $u$
be the weak solution of the mixed problem with Neumann data an atom
$a$ and zero Dirichlet data. The result for data in $H^1(N)$ follows
easily from the result for an atom.

We establish a representation for the gradient of $u$ in terms of the
boundary values of $u$.  Let $x \in \Omega$ and $j$ be an index
ranging from $1$ to $n$. We claim 
\begin{eqnarray} \label{RepFormula}
\frac {\partial u } { \partial x_j} (x) 
& = & 
\int _{ \partial \Omega }   \sum _{ i=1 } ^ n \frac { \partial \funsol
}{ \partial y_i } (x-\cdot)(  \nu _i \frac { \partial u }{ \partial y_j }  
- \frac { \partial u }{ \partial y _i } \nu _j  )
  \nonumber \\
& & \qquad + \frac {\partial \funsol }{\partial y_j } (
x-\cdot) \frac { \partial u }{ \partial \nu  } \, d\sigma .
\end{eqnarray}
If $u$ is smooth up to the boundary, the proof of (\ref{RepFormula})  is a straightforward
application of the divergence theorem. However, it takes a bit more
work to establish this result when we only have that $u$ is a weak
solution.

Thus, we suppose that $\eta $ is a smooth function
which is zero in a neighborhood of $ \crease $ and supported in a
coordinate cylinder. Using the coordinate system for our coordinate cylinder,
we set $ u_ \tau (y) = u(y+\tau e_n)$ where $e_n$ is the unit vector
the $x_n$ direction and $ \tau >0$. Applying the divergence theorem
gives 
\begin{eqnarray} 
\lefteqn{
\int _{ \partial \Omega }\eta  ( \frac { \partial \funsol }{ \partial \nu 
} (x- \cdot ) \frac { \partial u_\tau  }{ \partial y_j } - \nabla
\funsol (x-\cdot )
\cdot \nabla u_\tau   \nu _j  
 + \frac {\partial \funsol }{\partial y_j } (
x-\cdot ) \frac { \partial u_\tau  }{ \partial \nu  } ) \, d\sigma  } 
~~~~~~~~~~~~~~~~~~~~~~~~~~~~~~~~~
\nonumber\\
& =  &  \eta (x) \frac {\partial u_\tau  }{\partial x_j } (x) 
+\int _\Omega \nabla \eta  \cdot \nabla  \funsol(x- \cdot ) \frac {
  \partial u_\tau  }{ \partial y_j}   \nonumber \\
& & \qquad- \nabla _y \funsol (x-\cdot )  \cdot \nabla u_\tau  \frac {
  \partial \eta   } {\partial y_j}
\nonumber\\
& & \qquad + \frac { \partial \funsol }{ \partial y_j }
(x- \cdot )  \nabla u_\tau  \cdot \nabla \eta  \, dy . 
\label{dadgumidentity}
\end{eqnarray}
Thanks to the truncated maximal function estimate in Lemma
\ref{Whitney}, 
 we may let $ \tau $ tend to zero from above and conclude that the
same identity holds with $u_\tau$ replaced by $u$. 
Next, we suppose that $ \eta$ is of the form $ \eta \phi_ \epsilon$
where $ \phi_ \epsilon =0$  on    $\{ x: \delta (x) < \epsilon\}$, 
 $ \phi_ \epsilon =1$ on    $\{ x: \delta (x) >
2\epsilon \}$ and we have the estimate $ |  \nabla \phi_ \epsilon (x) |
\leq C/\epsilon $. Since we assume the boundary between $D$ and $N$ is
a Lipschitz surface, we have the following estimate for $ \epsilon $
sufficiently small
\begin{equation} \label{creasecollar}
|\{ x: \delta (x) \leq 2 \epsilon \}| \leq C \epsilon ^2.
\end{equation} 
Using our estimate for $ \nabla \phi_\epsilon$ and the inequality
(\ref{creasecollar}), we have 
$$
| \int _\Omega \eta   \nabla \phi_\epsilon    \cdot \nabla
\funsol(x-\cdot) \frac {   \partial u  }{ \partial y_j}  \,d y  |
\leq C \left ( \int_ { \{ y: \delta (y) < 2\epsilon \}
}   |\nabla u  | ^ 2\, dy \right ) ^ { 1/2} 
$$
and the last term tends to zero with $\epsilon$ since the gradient
of a weak solution lies in $L^2 ( \Omega)$.  Using this and similar
estimates for the other terms in (\ref{dadgumidentity}), gives 
\begin{eqnarray*} 
\lefteqn{ \lim _{ \epsilon \rightarrow 0^+}
\int _\Omega \nabla( \phi_\epsilon   \eta )  \cdot \nabla_y 
\funsol(x- \cdot ) \frac {   \partial u  }{ \partial y_j} 
- \nabla _y \funsol (x- \cdot) \cdot \nabla u  \frac { \partial (\phi_
  \epsilon \eta )   } {\partial y_j} }
~~~~~~~~~~~~~~~~~~~~~~~~~~~~~~~~~~~~~~~~~~~~~~~~~~~~~~~~~
& & \\
 + \frac { \partial \funsol }{ \partial y_j }
(x-\cdot )  \nabla u  \cdot \nabla(\phi_\epsilon  \eta )  \, dy  
&= & 
\int _\Omega \nabla \eta \cdot \nabla_y 
\funsol(x-\cdot ) \frac { \partial u  }{ \partial y_j} 
 \\
& & \qquad - \nabla_y \funsol (x- \cdot ) \cdot \nabla u  \frac { \partial \eta
} {\partial y_j} \\
& &\qquad + \frac {\partial \funsol }{ \partial y_j }
(x-\cdot )  \nabla u  \cdot \nabla \eta  \, dy .  
\end{eqnarray*}
Thus we obtain the identity (\ref{dadgumidentity}) with $ u_\tau$
replaced  by $u$  and
without the support restriction on $\eta$. 
Finally, we choose a partition of unity which consists of functions
that are either supported in a coordinate cylinder, or whose  support does
not intersect the boundary of $\Omega$. Summing as $\eta $ runs over
this partition gives us the  representation formula (\ref{RepFormula})
for $u$.     As we have $ \nabla u \in L^ q ( \partial
\Omega)$ for some $q>1$,  it follows from the theorem of Coifman, McIntosh and Meyer
\cite{CMM:1982}  that $ \nontan { ( \nabla u )} $ lies in $L^ q (
\partial \Omega)$. However, a bit more work is needed to obtain the
correct $L^1 $ estimate for $ \nontan { ( \nabla u )} $. 

We claim 
\begin{eqnarray*}
\int_{ \partial \Omega } \frac { \partial u}{\partial \nu } \, d\sigma
&=&0 \\
\int_{ \partial \Omega } \nu _j \frac{  \partial  u }{ \partial y_i }
- \nu _i \frac{  \partial  u }{ \partial y_j } \, d\sigma & = & 0. 
\end{eqnarray*}
Since  $ \nontan { ( \nabla u ) } $ lies in $L^q( \partial
\Omega)$, the proof of these two identities is a standard application
of the divergence theorem.   Using these results and the
estimates for $\nabla u$ in Theorem \ref{AtomicTheorem}, we can 
show that $\partial u /\partial \nu$ and $ \nu_j \partial u /\partial
y_i-\nu _i \partial u/\partial y_j$ are molecules on the boundary (see
\cite{CW:1976}) and
hence it follows from the representation formula (\ref{RepFormula})
that $\nontan{(\nabla u )}$ lies in $L^1 ( \partial \Omega)$ and
satisfies the estimate
$$
\| \nontan {(\nabla u) }\|_{L^1 ( \partial \Omega)}\leq C.
$$

Finally, the existence of non-tangential limits at the boundary
follows from the estimate for the non-tangential maximal
function. Once we know the limits exist it is easy to see that the
boundary data for the $L^1$-mixed problem must agree with the boundary
data for the weak formulation. 
\end{proof}

\section{Uniqueness of solutions} 
\label{Unique}

In this section  we establish uniqueness of solutions to the $L^1$-mixed
problem (\ref{MP}). We use the  existence result 
 established in section \ref{Atoms} and argue by duality that if  $u$ is a
solution of the mixed problem with zero Dirichlet and Neumann data,
then $u$ is also a 
solution of the regularity  problem with zero data and hence is zero. 
\begin{theorem} \label{uRuniq}
Let $ \Omega$, $N$ and $D$ be a standard domain for the mixed
problem. 
Suppose that $u$ solves the $L^1$-mixed problem (\ref{MP}) with data $ f_N =
0$ and $ f_D=0$. If $ ( \nabla u ) ^ * \in L^ 1 ( \partial \Omega)$,
then $u = 0$. 
\end{theorem}

Given a Lipschitz domain $ \Omega$, we may construct a sequence of
smooth approximating domains. A careful exposition of this
construction may be found in the dissertation of  Verchota
(\cite[Appendix A]{GV:1982}, \cite[Theorem 1.12]{GV:1984}). We will
need this approximation scheme and a few extensions.  Given a
Lipschitz domain $ \Omega$, Verchota constructs a family of smooth
domains $ \{ \Omega_k \}$ with $ \bar \Omega _k \subset \Omega$. In
addition, he finds bi-Lipschitz homeomorphisms $ \Lambda _k : \partial
\Omega   \rightarrow \partial \Omega _k$ which are constructed as
follows.

We choose a smooth vector field $ V$ so that for some $ \tau  =
\tau (M) >0$,  $ V\cdot \nu \leq -
\tau $ a.e.~on $\partial \Omega$ and define a flow $ f( \cdot ,\cdot ) :
\reals ^ n \times  \reals   \rightarrow \reals ^ n$   by 
$ \frac d { dt} f(x,t) = V(f(x,t))$, $f(x,0) = x$. One may find 
$ \xi > 0$ so that 
\begin{equation}\label{Gdef}
{\cal O} = \{ f(x, t) : x \in \partial \Omega , - \xi < t <
\xi \}
\end{equation} 
is an open set and the map $ (x,t) \rightarrow f(x,t)$ from $
\partial \Omega \times ( -\xi, \xi) \rightarrow {\cal O}$  is
bi-Lipschitz. 
Since the vector field $V$ is smooth, we have
\begin{equation} 
\label{DThing}
Df(x,t) = I _n + O(t)
\end{equation}
where $ I_n$ is the $n\times n$ identity matrix and $DF$ denotes the
derivative of a map $F$.  In addition, we have a Lipschitz function
$t_k (x)$ defined on $ \partial \Omega$ so that $ \Lambda_k (x) = f(x,
t_k ( x)) $ is a  bi-Lipschitz homeomorphism,  $ \Lambda _k : \partial
\Omega \rightarrow \partial \Omega _k$.  We may find  a collection of
coordinate cylinders $\{ Z _i \}$  so that each $ Z_i$ serves as a
coordinate cylinder for $ \partial \Omega$ and for each of the approximating
domains $ \partial \Omega_k$. If we fix a coordinate cylinder $Z$,   we have functions $
\phi$ and $\phi_k$ so that 
$\partial \Omega \cap Z =  \{ ( x',  \phi  ( x' )) : x' \in \reals ^ {
  n -1} \} \cap Z$ and 
$\partial \Omega_ k  \cap Z =  \{ ( x'  , \phi_k( x' )) : x' \in \reals ^ {
  n -1} \} \cap Z$. The functions $ \phi _k $ are $ C^ \infty $ and $
\|\pnabla \phi _k \| _ { L ^ \infty ( \reals ^ { n -1} )} $ is
bounded in $k$,  $ \lim _{ k \rightarrow \infty } \pnabla \phi _k (
x' ) = \pnabla \phi (x' )$ a.e.  and $ \phi_k$ converges to $\phi$
uniformly. Here we are using $ \pnabla$ to denote the gradient on
$\reals ^ {n-1}$. 

We let $ \pi: \reals ^ n \rightarrow \reals ^ { n-1}$ be the
projection $ \pi (x', x_n )  = x' $ and define $ S_k (x'
) = \pi(\Lambda_k (x' ,\phi (x')))$. 
According to Verchota, the map $S_k $ is bi-Lipschitz and has
a Jacobian which is bounded away from 0 and $ \infty$. 
We let $T_k$ denote $S_k^ { -1} $  and assume that both are defined in
a neighborhood of $ \pi(Z)$. 
 We claim  that 
\begin{equation}
\label{Dclaim} 
\lim _ { k \rightarrow \infty } DT_k (S_k(x') ) = I _ { n-1}, \qquad
\mbox{a.e.~in $\pi(Z)$}, 
\end{equation}
and  the sequence   $\| DT_k\|_ { L ^ \infty ( \pi (Z) )}$ is
bounded in $k$. 

To establish (\ref{Dclaim}), it suffices to show that $ DS_k $
converges to $ I_{ n-1}$ and that the Jacobian determinant of $ DS_k$
is bounded away from zero and infinity. The bound on the Jacobian is
part of Verchota's construction (see \cite[p.~119]{GV:1982}). As a
first step, we compute the derivatives of $t_k (x', \phi (x'))$.
We first observe that
\begin{eqnarray*} 
\lefteqn{\frac { \partial } { \partial x_i } f((x', \phi(x') ) , t_k (x',
\phi (x' ))) } ~~~~~~ \\
 &  = & 
\frac { \partial f}{ \partial x_i } ( (x', \phi (x' )), t_k (x' , \phi
(x')))
+\frac { \partial \phi }{ \partial x_i }( x' )   \frac { \partial f}{
  \partial x_n } (( x', \phi (x' )), t_k (x' , \phi (x')))  \\
& & \qquad 
+ V(f((x', \phi(x') ) , t_k (x',
\phi (x' ))))
\frac \partial { \partial x_i } t _k (x' , \phi (x' )). 
\end{eqnarray*} 
Since  
$ f((x', \phi  ( x' )), t_k (x', \phi (x')))$ lies in $ \partial
\Omega_k$, the derivative is tangent to $ \partial \Omega_k$ and  we have 
\begin{equation}\label{Tangential}
 \frac { \partial }{ \partial x_i } 
 f((x', \phi  ( x' )), t_k (x', \phi (x')))
\cdot \nu _ k (y) = 0, \qquad \mbox{a.e.~in }\pi (Z), 
\end{equation}
where $  y = ( S_k (x'), \phi_k ( S_k (x' )))$  and $ \nu _k $ is the
normal to $ \partial \Omega_k$. Solving  equation (\ref{Tangential})
for $ \frac \partial { \partial x_i } t_k $ gives 
\begin{eqnarray*} 
\frac \partial { \partial x_i } t _k (x' , \phi (x' ))
& = & - ( V(y) \cdot \nu _k (y) ) ^ { -1} \left ( \frac { \partial f }{ \partial
  x_i }(   ( x' , \phi ( x' )), t_k (x', \phi (x' )))  \right .  \\
& & \qquad 
+ \left. \frac { \partial \phi } { \partial x_i } (x' ) 
\frac { \partial f }{ \partial
  x_n }(   ( x' , \phi ( x' )), t_k (x', \phi (x' ))) \right)\cdot
\nu_k (y) .
\end{eqnarray*} 
Since $\lim _{ k \rightarrow \infty } t_k (x', \phi (x'))  = 0$
uniformly for  $x' \in  \pi 
(Z)$,  (\ref{DThing}) holds,   and $ \nu _ k ( y) $ converges
pointwise a.e.~and boundedly to $ \nu (x)$, we obtain that 
\begin{equation} 
\label{tderiv}
\lim _ { k \rightarrow \infty } 
\frac \partial { \partial x_i } t _k (x' , \phi (x' )) = 0, \qquad
\mbox{a.e.~in $ \pi (Z)$}.
\end{equation}
Given (\ref{DThing}),   (\ref{tderiv}), and recalling that $ S_k( x')
= \pi(f( (x', \phi(x')), t_k (x', \phi (x')))$,  (\ref{Dclaim})
follows.

\note
{
The proof of this Lemma is in RMB notebook 19, page 49.
}

\begin{lemma} 
\label{Verchota}
Let $ \Omega$, $N$ and $D$ be a standard domain for the mixed problem.
If $u$ is in $\sobolev 1 1 (\partial \Omega _k)$ and $w$ is the weak
solution of the mixed problem with Neumann data an atom for $N$ and
zero Dirichlet data, then we have
$$
\int _ { \partial \Omega _k } u \frac { \partial w }{ \partial \nu }
\,d\sigma \leq C_w \| u \|_{ W^ { 1,1 }( \partial \Omega _k )}.
$$
\end{lemma}

\begin{proof} 
This may be proven using generalized Riesz transforms as in
\cite[Section 5]{GV:1984}.  Also, see more recent treatments by Sykes
and Brown \cite[section 3]{SB:2001} and Kilty and Shen
\cite[section 7]{KS:2010}.  Verchota's argument uses square function
estimates to show that the generalized Riesz transforms are bounded
operators  on
$L^p ( \partial \Omega)$. In the proof of this Lemma, we need that the Riesz transforms
of $w$ are bounded functions.  From the estimate for the Green function in  Lemma
\ref{Green} and the representation of $ w= - \langle G,
a\rangle_{\partial \Omega} $, we
conclude that $ w$ is H\"older continuous. The H\"older continuity,
and hence boundedness,  of the Riesz transforms of $w$ follow from
the following characterization of H\"older continuous harmonic
functions. A harmonic function  $u$ in a Lipschitz domain  $\Omega$ is H\"older
continuous  of exponent $\alpha$, $0< \alpha < 1$,  if and  only if  $ \sup _{ x\in \Omega } \dist(x, \partial
\Omega )^ {1-  \alpha } |\nabla u (x) | $ is finite. 
\end{proof} 

We will need the following  technical lemma on
approximation of functions  with $ \nontan{(\nabla u )} $ in $L^ 1 (
\partial \Omega)$. The proof relies on the approximation scheme of
Verchota
outlined above. In our application,  we are
interested in studying  functions in Sobolev spaces on the family of
approximating domains.  Working with derivatives makes the argument fairly
intricate. 
\begin{lemma} 
\label{TechnicalMonstrosity}
Let $\Omega$, $N$ and $D$ be a standard domain for the mixed problem. 
If $ u $ satisfies $ \nontan { ( \nabla u ) } \in L^ 1 (
  \partial \Omega)$  and $ \nabla u$ has non-tangential limits a.e.~on
$ \partial \Omega$, then we may find  a sequence of Lipschitz
functions $ U_j$ so that 
$$
\lim _{ k\rightarrow \infty } \| u - U _j \| _ { \sobolev 1 1 (
  \partial \Omega _k )} \leq C /j.
$$
If the non-tangential limits of $u$ are zero a.e. on $D$, then we may
arrange that $U_j|_{\partial \Omega}$ is zero on  $ D$. 

The constant $C$ may depend on $  \Omega$ and $u$.  
\end{lemma} 

\begin{proof} To prove the Lemma, it suffices to consider a function
  $u$ which is zero outside one of the coordinate cylinders $Z$ as
  given in Verchota's approximation scheme. 
We have $ u( x' , \phi (x'))  \in \sobolev 1 1 ( \reals ^ { n
  -1})$, where we have set this function to be zero outside $ \pi (Z)$. Hence, there exists a sequence of Lipschitz functions $ u _j $
so that $\int _{ \reals ^ { n-1} } | \pnabla u (x', \phi (x' )) -
\pnabla u _ j ( x',  \phi (x' )) | \, dx' \leq 1/j$ where $ \pnabla $
denotes the gradient in $ \reals ^ { n -1}$. We extend $u_j$ to
a neighborhood of $ \partial \Omega$ by 
$$U _ j ( f(x,t) ) = \eta (
f(x, t) ) u _j (x) , \qquad x \in \partial \Omega$$
where $ \eta$  is a smooth cutoff function which is one  on  a
neighborhood of  $ \partial
\Omega$ and supported in  the set $ {\cal O} $ defined in
(\ref{Gdef}).  
If we have that $u$ is zero on $D$, then standard approximation
results for Sobolev spaces allow us to choose $u_j$ to be zero in 
a neighborhood of $\bar D$. 
This relies on our assumption on $ \Lambda $. 

We consider
\begin{eqnarray*} 
\lefteqn{ 
\int _ {\pi ( Z)} | \pnabla u (x' , \phi _k (x' )) - \pnabla U_j ( x',
\phi_k (x' ))| \, dx' } 
~~~~~~~~~~~~~~~~~~~~
& &  \\
& 
\leq & 
\int _ {\pi ( Z)} | \pnabla u (x' , \phi _k (x' )) - \pnabla u ( x',
\phi (x' ))| \, dx' \\
& & \qquad +
\int _ {\pi ( Z)} | \pnabla u (x' , \phi  (x' )) - \pnabla u_j ( x',
\phi (x' ))| \, dx'\\
& & \qquad +
\int _{ \pi ( Z)} | \pnabla u_j  (x' , \phi  (x' )) - \pnabla U_j ( x',
\phi_k (x' ))| \, dx'\\
& = & A_k +B +C_k.
\end{eqnarray*} 
We have that $ \lim _{ k \rightarrow \infty } A_k =0 $ since we assume
that $ \nontan {(\nabla u )} \in L ^ 1 ( \partial \Omega )$, $ \nabla
u $ has non-tangential limits a.e., and $ \pnabla \phi _k $ converges
pointwise a.e.~and boundedly to $ \pnabla \phi$. By our choice of $ u_j$,
we have $ B\leq C/j$. Finally, our construction of $U_j$ and our
definition of $T_k$ (before (\ref{Dclaim})) imply that  $ U_j (x', \phi _k ( x' )) = u _j
( T_k (x') , \phi ( T_ k (x')))$ and hence we have
\begin{eqnarray*} 
C_k & \leq  &
\int _ { \pi (Z)} | 
( I _ { n-1} - DT_
k ( x' ))  \pnabla u _ j ( x', \phi (x'))  | \, dx'  \\
& & \qquad  
+ \int _ { \pi (Z)}  | DT_k (x' ) ( \pnabla  u _ j ( x' , \phi(x' ))) -
\pnabla u _ j ( T_k (x' ), \phi ( T_k (x' )))| \, d x' 
\\
& = & C_ { k, 1 } + C_ {k,2 } .
\end{eqnarray*}
We have that $ \lim _ { k\rightarrow \infty } C_ { k,1 } = 0$ since $
\pnabla u _ j $ is bounded and  (\ref{Dclaim}) holds. 
  Since $ T_k ( x' )$
converges uniformly to $x' $,  $DT_k $ is bounded  and the Jacobian of
$S_k$ is bounded, we have that $ \lim _
{ k \rightarrow \infty } C_{ k,2} = 0$. 
\end{proof}

\begin{proof}[Proof of Theorem \ref{uRuniq}]
We let $ u$ be a solution of the $L^1$-mixed problem,  (\ref{MP}),
with $f_N = 0$ and $f_D=0$ 
and we wish to show that $u$ is zero. We fix $a$ an atom for $N$ and
let $ w$ be  a solution of the mixed problem with Neumann data $a$ and
zero Dirichlet data.  
Our goal is to show that 
\begin{equation}\label{AtomClaim}
\int _ { N } a u \, d\sigma =0.
\end{equation}
This implies that $u$ is zero on $ \partial \Omega$ and then Dahlberg
and Kenig's result for uniqueness of  solutions of the  regularity problem
\cite{DK:1987} implies  that $u=0$ in $ \Omega$. 

We turn to the proof of (\ref{AtomClaim}). Applying Green's second
identity in one of  the approximating domains  $ \Omega _k$ gives us 
\begin{equation}\label{uniq42}
\int _ { \partial \Omega _k } w \frac { \partial u }{ \partial \nu} \,
d\sigma =
\int _ { \partial \Omega _k } u \frac { \partial w }{ \partial \nu} \,
d\sigma, \qquad k =1,2\dots .
\end{equation}
We have  $ \nontan { ( \nabla u )} $ is in $L^1 ( \partial \Omega )$
while  $
w$ H\"older continuous and hence bounded. Recalling that $w$ is zero
on $D$ and $ \partial u /\partial \nu$ is zero on $N$, we may use the
dominated convergence theorem to obtain 
\begin{equation}\label{uniq43}
\lim _ { k \rightarrow \infty } \int _ { \partial \Omega _k } w \frac
     { \partial u }{ \partial \nu } \, d\sigma =0 .
\end{equation}
Thus, our claim will follow if we can show that 
\begin{equation}
\label{ClaimFollow}
\lim _{ k \rightarrow \infty } \int _ { \partial \Omega _k } u \frac { \partial w }{ \partial \nu }
\, d\sigma   =  \int _{\partial \Omega } ua \, d\sigma .
\end{equation}
Note that the existence of the limit in (\ref{ClaimFollow}) follows
from (\ref{uniq42}) and (\ref{uniq43}).
We let $U_j$ be the sequence of functions from Lemma \ref{TechnicalMonstrosity}
and consider
\begin{eqnarray}
\lefteqn{
\left | \int _{ \partial \Omega } ua \, d\sigma 
-\lim _{ k \rightarrow \infty}  \int _{ \partial \Omega _k } u \frac { \partial w }{ \partial \nu }
\, d\sigma \right | }  \label{U1}\\  & \leq & 
\left| \int_{ \partial \Omega }  ua \, d\sigma - \lim _{ k \rightarrow
  \infty }  \int _{ \partial
  \Omega _k } U_j \frac { \partial w }{ \partial \nu }  \, d\sigma
\right | 
 + 
\limsup _{ k \rightarrow \infty } \left | \int _ { \partial \Omega _k }  ( u -
U_j ) \frac { \partial w}{ \partial \nu } \, d\sigma \right| .
\nonumber
\end{eqnarray} 
Because we have that $ \nontan {( \nabla w ) } $ is in $L^ 1( \partial
\Omega)$ and $U_j$ is bounded, we may take the limit of the first term
on the right of  (\ref{U1}) and obtain 
$$
\left |  \int_{ \partial \Omega }  ua \, d\sigma - \lim _{ k \rightarrow
  \infty }  \int _{ \partial
  \Omega _k } 
U_j
 \frac { \partial w }{ \partial \nu }  \, d\sigma
\right | 
=
\left |  \int _{ N  } ( u -U _j ) a 
\, d\sigma\right | 
\leq C/j .
$$
Here we use that $U_j|_{ D } =0$. 
According to Lemmata \ref{Verchota} and \ref{TechnicalMonstrosity},
the second term on the right of (\ref{U1}) is bounded by $ C_w/j$. As
$j$ is arbitrary, we obtain (\ref{ClaimFollow}) and hence the Theorem. 
\end{proof} 

\section{A Reverse H\"older inequality at the boundary} 
\label{BoundaryReverse}
In this section we establish an estimate in $L^p( \partial \Omega)$
for the gradient of a solution to the mixed problem.  This is the key
estimate that is used in section \ref{LpSection} to establish
$L^p$-estimates for the mixed problem.

\begin{lemma}
\label{newLocal}
 Let $ \Omega$, $N$ and $D$ be a standard domain for the
  mixed problem. Let $u$ be a weak solution of the mixed problem with
Neumann  data $f_N$ in $L^ \infty ( N)$  and zero Dirichlet data. 
Let $ p_0 >2$ be as in Lemma \ref{RHEstimate} and fix $ q$ satisfying
$ 1< q < p_0/2$. For $ x \in \partial \Omega $ and $ r$ with $ 0 < r
< r_0$ we have 
$$
\left ( \average _ { \sball x r } { \nontan {( \nabla u  ) _{cr} } }^ q
  \, d\sigma \right) ^ { 1/q} 
 \leq   C \left [ 
 \average _{ \dball x { 2r}   }   |  \nabla u  | 
\, dy 
 +  \|f_N\|_{ L^ \infty ( \sball x { 2r} \cap N)} \right].
$$
The constant $c=1/16$ and  $C$ depends on $M$, $n$ and $q$. 
\end{lemma}

\begin{proof}
We fix $ x\in \partial \Omega$ and $r$ with $0< r < r_0$. We claim
that we have 
\begin{equation}\label{LocalClaim}
\left ( \average _ { \sball x {4r}} |\nabla u |^ q \, d\sigma \right )
  ^ { 1/q} \leq C \left ( \average _ { \dball x {16r}} |\nabla u | \,
    dy + \| f_N\| _ { L^ \infty ( \sball x { 16r} \cap N)} \right).
\end{equation}
We will consider two cases: a) $ \delta (x) \leq 8r \sqrt { 1+M^2}$ ,
b) $ \delta (x)  > 8r \sqrt { 1+M^2} $.
We give the proof in case a). Since we assume $ 1 < q < p_0/2$, we may
choose $ \epsilon $ satisfying $ 2-2/q < \epsilon < 2 - 4 /p_0$. We
apply H\"older's inequality with exponents $ 2/q$ and $2/(2-q)$ to 
obtain 
\begin{eqnarray*}
\lefteqn{ \left ( \int _ { \sball x { 4r}} |\nabla u | ^ q \, d\sigma \right ) ^
  { \frac 1 q}  } 
\qquad 
\\
& \leq &
   \left ( \int _ { \sball x  { 4r} } |\nabla u
    |^ 2 \delta ^ { 1-\epsilon } \, d\sigma \right ) ^ { \frac 1 2 }
\left ( \int _ {\sball x { 4r} } \delta ^ { ( \epsilon -1)q /( 2-q)}\,
  d\sigma \right ) ^ { \frac 1 q - \frac 1 2 }  \\
& \leq  &
C r ^ { ( n-1) ( \frac 1 q - \frac 1 2 ) + \frac { \epsilon -1 }
  2 }\left ( \int _ { \sball x { 4r} } |\nabla u | ^ 2 \delta ^ {
  1-\epsilon } \, d\sigma \right ) ^ {1/2}
\end{eqnarray*}
where we use that $ q (\epsilon -1) / ( 2-q) > -1$ or $ 2 - \frac 2 q
< \epsilon $ which implies that the integral of $ \delta ^ { (
  \epsilon -1) q /( 2-q)}$ is finite.  Next, we use Lemma
\ref{Whitney} and our hypothesis that  $ \delta (x) \leq 8r \sqrt {
  1+M^2}$   to obtain 
\begin{eqnarray*}
\lefteqn { \left ( \int _ { \sball x { 4r} } |\nabla u | ^ 2 \delta ^ { 1-
  \epsilon } \, d\sigma \right ) ^ { 1/2 }  } \\
& \leq &  C
\left [ 
\left ( \int _ {\dball
      x {8r} } |\nabla u |^ 2 \, \delta ^ { - \epsilon } \, dy
    \right ) ^ { 1/2} 
+ 
 \left ( \int _ { \sball x { 8r}\cap N} |f_N |^ 2 \delta ^ { 1-
      \epsilon } \, d\sigma \right ) ^ { 1/2} 
\right ] \\ 
& \leq &  C \left [
 \left ( \int _ { \dball x { 8r} } |\nabla u |^ 2\delta
  ^ { - \epsilon } \, dy \right ) ^ { 1/2 } 
+
r ^ { \frac { n-
      \epsilon } 2 } \| f_N \| _ { L^ \infty ( \sball x { 8r }  \cap N
    ) }
  \right ] .
\end{eqnarray*}
To estimate $(\int _ {\dball x {8r}}|\nabla u | ^ 2 \delta ^ {
  -\epsilon } \, dy  ) ^ {1/2}$, we choose $p >2$, use H\"older's inequality
with exponents $ p/2 $ and $ p / (p-2)$, and Lemma \ref{RHEstimate}
to 
find
\begin{eqnarray*}
\lefteqn{
\left( \int _ { \dball x { 8r}}|\nabla u | ^ 2 \delta ^ { -\epsilon }
\, dy \right ) ^ { 1/2}  }  \\
& \leq  & 
\left( \int _ { \dball x { 8r}}\delta ^ {-  \epsilon  p / ( p-2)} \,
dy \right ) ^ { \frac 1 2 - \frac  1 p }\left ( \int _ { \dball x { 8r
  } }|\nabla u | ^ p \, dy \right ) ^ { 1/p }  \\
& \leq &  C r ^ { - \frac \epsilon 2  + \frac n 2 } \left [ 
 \average _ { \dball x { 16r } }|\nabla u | \, dy 
+ \left ( \average _ { \sball x { 16r} \cap N}|f_N | ^ { \frac { p _0 ( n-1)}
  n} \, d\sigma \right ) ^ { \frac n { p_0 ( n-1) } } \right ].
\end{eqnarray*}
Combining the two previous displayed inequalities   gives the estimate
$$
\left( \int _ { \sball x { 4r} } |\nabla u | ^ q \, d\sigma \right ) ^
     { 1/q} 
\leq C r ^ { ( n-1)/q} \left ( \average_ { \dball x { 16r}} |\nabla u | \, dy + \| f _N \|
_ { L^ \infty ( \sball x { 16r} \cap N )}\right ), 
$$
which gives the claim (\ref{LocalClaim}). 

Now we consider the proof of (\ref{LocalClaim}) in case b). Here, we
use  $ \delta (x) >8r \sqrt{ 1+M^2}$ and Lemma \ref{DorN} to conclude
that $ \sball x {8r}\subset N$ or that $ \sball x {8r} \subset
D$. Then we may use 
Lemma \ref{NeumannRegularity} or Lemma \ref{DirichletRegularity} to
conclude that 
$$
\int _ { \sball x { 4r}} |\nabla u | ^ 2 \, d\sigma \leq C \left( 
\int _ { \sball x { 8r} \cap N} |f_N |^ 2 \, d\sigma 
+ \frac 1 r \int_{\dball x { 8r} } |\nabla u | ^ 2 \, dy \right ).
$$
Next, Lemma \ref{RHEstimate} gives 
$$
\left ( \average  _ { \dball x { 8r}} |\nabla u | ^ 2 \, dy
\right ) ^ { 1/2} 
\leq
C \left (    \average _ { \dball x { 16r} }  |\nabla u
| \, dy + \| f\| _ { L^ \infty ( \sball x { 16r} \cap N) } \right). 
$$
Using the  two previous  estimates and H\"older's inequality,  we obtain 
the claim (\ref{LocalClaim}) in case b).

To obtain the estimate for the non-tangential maximal function, we
choose a cutoff function $ \eta $ which is 
one on $ \ball x { 3r}$ and supported in $ \ball x { 4r}$. 
By repeating the arguments in the proof of Theorem \ref{HardyTheorem},
we may show that for $z $ in $ \Omega$ and $j =1,\dots,n$, we have the
following representation for the derivatives of $u$: 
\begin{eqnarray*}
(\eta\frac {\partial  u}{\partial z_j}) (z) 
 &  = &  \int _ { \partial
    \Omega } 
\eta ( \frac {\partial \funsol }{ \partial \nu } ( z-\cdot ) \frac {
  \partial u }{\partial y _j } 
- \nu _j  \nabla_y \funsol (z- \cdot ) \cdot \nabla u
+ \frac { \partial \funsol }{\partial y _j }(z-\cdot) \frac  {
  \partial u }{ \partial \nu }  ) \, d\sigma \\
& & 
\qquad - \int _ { \Omega} \nabla \eta \cdot \nabla_y \funsol ( z-\cdot )\frac {
  \partial u }{ \partial y _j } 
- \frac { \partial \eta }{ \partial y _j } \nabla_y \funsol (
z-\cdot)\cdot \nabla u  \\
& & \qquad \qquad  \qquad + \nabla \eta \cdot \nabla u \frac { \partial
  \funsol }{ \partial y _j } ( z- \cdot ) \, dy .
\end{eqnarray*}
From this representation and the theorem of Coifman, McIntosh and
Meyer \cite{CMM:1982}, we obtain
$$
\left ( \average _ { \sball x r } { \nontan {( \nabla u  ) _{r} } }^ q
  \, d\sigma \right) ^ { 1/q} 
\leq C \left [ \average_{ \dball x { 4r}} |\nabla u | \, dy + 
\left( \average _ { \sball x { 4r}}|\nabla u | ^ q \, d \sigma \right) ^ {
  1/q}\right]. 
$$
From this estimate, the claim (\ref{LocalClaim})  and a covering
argument, we obtain  the Theorem. 
\end{proof}

\section{Estimates for solutions with data from   $L^p$, $p> 1$}
\label{LpSection}
In this section, we use the following variant of an argument developed
by Shen \cite{ZS:2007} to establish    $L^p$-estimates for elliptic problems in
Lipschitz domains. Shen's argument is based on earlier in work of
Caffarelli and  Peral   \cite{MR1486629}. 

As the argument depends on a  Calder\'on-Zygmund decomposition into 
dyadic cubes, it will 
be stated using  surface cubes  rather than  the surface balls $ \sball xr $
used elsewhere in this paper. 

Let $Q _0$ be a cube in the boundary and let $F$ be defined on $4
Q_0$. Let the exponents $ p $ and $q$ satisfy $ 1< p < q$.  Assume that for each $Q \subset Q_0$, we may find two functions
$F_Q$ and $R_Q$ defined in $2Q$ such  that 
\begin{eqnarray} 
\label{Shen1} 
|F|  & \leq &  |F_Q| + |R_Q | ,   \\
\average_{ 2Q} |F_Q| \, d\sigma &\leq & C \left( \average_{ 4Q} |f|^ p \, d\sigma
\right ) ^ {1/p},   \label{Shen2}  \\
\left ( \average_{2Q}  |R_Q|^ q \, d\sigma \right) ^ {1/q} &\leq & 
C \left [   \average _{ 4Q}  |F|\,d \sigma   + \left( \average_{4Q} |f|^ p \,
  d\sigma \right) ^ {1/p} \right]. \label{Shen3}
\end{eqnarray}
Under these assumptions, for $r$ in the interval $ ( p, q)$, we have
$$
\left( \average _{Q_0} |F|^ r \, d\sigma \right)^ { 1/r } \leq C \left[
\average _{4Q_0} |F|\, d\sigma + \left( \average_{4Q_0}  |f|^ r \, d\sigma
\right ) ^ { 1/r} \right ] .
$$ 
The constant in this estimate will depend on the Lipschitz constant
of the domain, the $L^p$ indices involved and the constants in the
estimates in the conditions (\ref{Shen2}--\ref{Shen3}).  The argument
to obtain this conclusion is more or less the same as in Shen
\cite[Theorem 3.2]{ZS:2007}. The main differences arise because the
last term in (\ref{Shen3}) require us to substitute the maximal
function  $M(|f|^p)^ {1/p}$
for  $M(f)$. We omit a detailed proof. Our
hypotheses hypotheses differ from Shen's in that Shen has $p=1$ in
(\ref{Shen2}) and (\ref{Shen3}) while we have $p>1$. We need to change
Shen's formulation because we begin with results in Hardy spaces,
rather than $L^p$-spaces.

In our application, we will let $ 4Q_0$ be a cube with sidelength
comparable to $r_0$.  We let $u$ be a solution of the mixed problem
with Neumann data $f$ in $L^p(N)$ and Dirichlet data zero. We define
$f$ to be zero in  $D$.  Since
$L^p(N)$ is contained in the 
Hardy space $H^1(N)$, we may use  Theorem
\ref{AtomicTheorem} to obtain a solution of the mixed problem with
Neumann data $f$ on $N$ and zero Dirichlet data on $D$.  
Let $F = \nontan{(\nabla u )}$ and given a cube $Q\subset Q_0$ and
with diameter $r$,  define $ F_Q$ and $R_Q$ as follows. 
We let $\bar f _{ 4Q} =0$ if  $ 4Q \cap D \neq \emptyset $ and $ \bar f _{ 4Q}
= \average _{ 4Q} f \, d\sigma $ if $ 4Q \subset N$. Set $g = \chi _{
  4Q} ( f-\bar f _{ 4Q})$ and $h = f-g$. As both $g$ and $h$ are
elements of the Hardy space $ H^ 1(N)$, we may  use Theorem
\ref{HardyTheorem} to find solutions of  the $L^1$-mixed problem
with Neumann data $g$ or $h$. We let $v$ be the solution with Neumann
data $g$ and $w$ be the solution with Neumann data $h$.  According to
the uniqueness result Theorem \ref{uRuniq} we have $ u = v+w$.   
We let  $R_Q = \nontan { ( \nabla w )}$  and $ F_Q =
\nontan {( \nabla v ) } $ so that (\ref{Shen1}) holds. We turn our
attention to  establishing (\ref{Shen2}) and  (\ref{Shen3}). 

To establish (\ref{Shen2}),  observe that the 
$H^1$-norm of $g$ satisfies the bound   $$ \| g \|_{ H^1 (N)} \leq C \| f
\| _ {L^p( 4Q)}  \sigma (Q) ^ { 1/p'}. $$  With this,  the estimate
(\ref{Shen2}) follows from Theorem \ref{AtomicTheorem}. 
Now we turn to the estimate (\ref{Shen3}) for $ F_Q = \nontan {(
  \nabla w ) } $. We note that the Neumann data $h$ is constant on $
4Q \cap N$.
We define a  maximal  operator by   taking the supremum over 
that part of the cone that is far from the boundary, 
$$
\nontan { ( \nabla w )_+}(x) = \sup _{ y \in \ntar x \setminus \ball x
  {Ar} }    |\nabla w (y)| 
$$
where $A$ is to be chosen.   

A simple geometric argument gives that
\begin{equation} \label{far} 
\nontan { ( \nabla w )_+}(x) \leq C \average   _{ 4Q}  \nontan { (
  \nabla w ) } \, d\sigma , \qquad x \in  2Q.  
\end{equation} 
The estimate for $ \nontan { ( \nabla w )_{Ar}  }  $ uses the local
estimate for the mixed problem in Lemma \ref{newLocal}  to conclude
that 
\begin{eqnarray} 
\nonumber 
\left ( \average _ { 2Q} {  \nontan { ( \nabla w ) _ { Ar}  }} ^ q \, d\sigma
  \right) ^ { 1/q} 
& \leq  & 
C
\left [  \|h\|_{L^ \infty ( 4Q)} 
+ \average _ { T( 3Q) } |\nabla w | \, d\sigma \right ]  
\\
& \leq & 
C \left [  \left( \average_{4Q} |f|\, d\sigma \right) 
  + \average _ {4Q} \nontan { ( \nabla w ) } \, d \sigma \right ].
  \label{near}
\end{eqnarray} 
provided that  the constant $A$ in the definition  of  $ \nontan{ (\nabla w )_+}
$ is chosen sufficiently  small. 
Recall that  $ T(Q)$ was defined at the beginning of the proof of
Lemma \ref{Whitney}. 
From the estimates (\ref{far}) and (\ref{near}), we conclude that 
\begin{equation}
\label{New3} 
\left( \average _{ 2Q} ( R_Q ) ^ q \, d\sigma \right) ^ { 1/q} 
 \leq   C \left[  \average _ {4Q}  |f| \, d\sigma 
 +  \left( \average _{ 4Q} \nontan { ( \nabla w ) }\, d\sigma \right) ^
{ 1/p } \right ].   
\end{equation}
We have $ \nontan {( \nabla w )} \leq \nontan { ( \nabla v ) } +
  \nontan {( \nabla u ) } $ and hence we may estimate the term involving
 $ \nontan { ( \nabla w )}$ by 
$$
\average _ { 4Q} \nontan {( \nabla w ) } \, d\sigma 
\leq 
\average _ { 4Q} \nontan {( \nabla u ) } \, d\sigma +
\average _ { 4Q} \nontan {( \nabla v ) } \, d\sigma 
\leq 
\average _ { 4Q} \nontan {( \nabla u ) } \, d\sigma +
C \left ( \average _ { 4Q} |f| ^ p \, d\sigma \right ) ^ { 1/p}   
$$
where we have used Theorem \ref{HardyTheorem} to estimate the term
involving  $\nontan{(\nabla v)}$. 
Combining this with (\ref{New3}) gives  (\ref{Shen3}). 

Applying the  technique  of Shen outlined above  gives the
$L^p$-estimate and thus we obtain the following theorem. 

\begin{theorem}
Let $ \Omega$, $N$ and $D$ be a standard domain for the mixed
problem and let $p$ satisfy $ 1 < p < p_0/2$ where  $p_0$ is  from  Lemma \ref{RHEstimate}. 

Given data $f_N$ in $L^p(N)$, we may solve the $L^p$-mixed problem with
Neumann data $f_N$ and Dirichlet data 0 and this solution satisfies 
the estimate
$$
\| \nontan {( \nabla u )} \|_{ L^ p ( \partial \Omega )} 
\leq C \| f _N \| _{ L^  p (\partial \Omega )} .
$$
The constant $C$ depends on the global character of the domain and the
index $p$.  
\end{theorem}

\section{Further  questions}
This work adds to our understanding of the mixed problem in Lipschitz
domains.  However, there are several avenues which are not yet
explored.
\begin{enumerate}
\item Can we study the inhomogeneous mixed problem and obtain results
  similar to those of Fabes, Mendez and M. Mitrea \cite{FMM:1998} and
I.~Mitrea and  M.~Mitrea \cite{MM:2007}?
\item Is there an extension to  $p < 1$  as the work of Brown \cite{RB:1995a}?
\item Can we study the mixed problem for more general decompositions
  of the boundary, $ \partial \Omega = D \cup N$? To what extent is
  the condition that the boundary between $D$ and $N$ be a Lipschitz
  graph needed?
\item Can we extend these techniques to elliptic systems and higher
  order elliptic equations?
\end{enumerate} 

\note
{ Index of notation.

\begin{tabular}{rl}

\sc Symbol &  Meaning \rm \\
$\delta(x)  $ & the distance from $x$ to the crease $\crease$  \\
$D$ & region where we specify Dirichlet data\\
$N$ & region where we specify Neumann data \\
$P$& $P$ operator \\
$\Xi$  & fundamental solution \\
$\ntar x  $ & non-tangential approach region\\
$\crease $ & boundary between $N$ and $D$ \\
$\Omega $ & domain \\
$ \locdom x r $ & domain of size $r$ near a boundary point $x$.  \\
$\dball x r $ & $\Omega \cap \ball x r $ \\
$T(Q) $ & in proof of Lemma \ref{Whitney}  
$T(Q)   = \{ x \in \bar \Omega : \dist(x, Q) < \diam(Q)\}  $
\end{tabular}
}

\def\cprime{$'$} \def\cprime{$'$} \def\cprime{$'$}

\end{bibunit}



\renewcommand{\thesection}{\Alph{section}}
\renewcommand{\thetheorem}{\Alph{section}.\arabic{theorem}}
\renewcommand{\theequation}{\Alph{section}.\arabic{equation}}
\newenvironment{example}[1][Example]{\begin{trivlist}\item[\hskip \labelsep
{\it #1. }]}{  \goodbreak \end{trivlist}}   

\newpage

\begin{bibunit}

\appendix

\section{Correction to the Proof of Theorem 7.7.}\label{Introduction}

In an earlier work \cite[Equation (7.4)]{MR3042705} we make a
technical claim about 
non-tangential maximal functions which is an essential step in our
study of the $L^p$-mixed problem. 
The example below shows that the statement (7.4) is  incorrect. 
 A correct substitute  for this claim may  be
found in \eqref{Far} below. 
This note   provides a proof of the main theorem of
\cite{MR3042705} which avoids the problematic claim. We will follow the
notation, equation references, and results from  our earlier work.
We thank L. Croyle for pointing out this error. Her
dissertation  \cite{MR3564069} includes a different approach to
correcting this error and the paper \cite{BC:2019}  includes a version of the
argument presented here. 
The
correction outlined here should also be applied to several subsequent
papers including \cite{MR3040944,MR3034453}
that make use of the method from  \cite{MR3042705}.

We begin with an example which shows that the claim (7.4) may fail. 

\begin{example} Consider the  domain  $\Omega$  in $ \reals ^2$ which lies
  above  the graph of the Lipschitz function $ \phi : \reals
  \rightarrow \reals $ given by $ \phi (x_1) = \max ( -2|x_1|, 2|x_1|
  -4)$. For this domain, we may find $ y \in \Gamma (0)$ with $|y|$
  large, and $  \epsilon = (1+\alpha ) \dist (y ,
  \partial \Omega) -|y|$  small. This requires that $ \alpha $ be
  sufficiently large. We claim that we have $ \{ x : y \in \Gamma (x),
  x\in \sball 0 { 1/2} \} \subset \sball 0 {C\epsilon } $. To see
  this, note that if $ x\in \sball 0 { 1/2} $, then we have $ |x- y|
  \geq |y| + c|x|$. This follows because the line segments that form
  the boundary near zero are not tangent to the boundary of the ball
  $\partial B_{|y|}(y)$. Thus $y \in \Gamma (x)$ will hold only if $|x|$
  is at most a multiple of $ \epsilon$.  If we let $ u (z) = \chi _ {
    B_ \epsilon (y)}(z)$, then we have $ u^ * (0) =1$, but $ \average
  _ { \sball 0 { 1/2}} u ^ * \, d\sigma \leq C \epsilon$. Thus (7.4)
  fails.
\end{example}

To avoid making use of (7.4) in Ott and Brown, we apply Shen's argument
outlined in (7.1-3) to  $M((\nabla u^*)^ {1/2})^2$ rather than to $\nabla u^*$. 
Here, $M$ is the Hardy-Littlewood maximal function which we define by
$$
M(f)(x) = \sup _{ s >0 } \average _ { \sball x s } |f|\, d\sigma .
$$
We will need several auxiliary maximal functions which we define
here. For these definitions we need a parameter $r$ which will
give a division between small and large scales.  In our applications, 
$r$ will be comparable to the sidelength of the cube $Q$ that appears
in Shen's argument outlined in (7.1-3).   We let
$$
M_ 0 (f) (x) = \sup _{0<s < r }  \average _{ \sball x s } |f| \, d\sigma,
\qquad
M_ \infty (f) (x) = \sup _{s \geq  r }  \average _{ \sball x s } |f| \, d\sigma.
$$
We also define  truncated non-tangential maximal functions by
$$
u^ \triangledown (x) = \sup _{y \in \Gamma (x), |x-y |< cr} | u(y) |,
\qquad
u^ \vartriangle (x) = \sup _{y \in \Gamma (x), |x-y |> cr} | u(y) |.
$$

We define $ F= M (( \nabla u ^ *) ^ { 1/2} ) ^2$ and   fix a cube
$Q$.
\comment{We let
$$
\bar f_ {  Q } = \left \{ \begin{aligned} & 0, \qquad &  Q \cap \bar D \neq \emptyset
  \\
  & \average _Q f \,d \sigma \qquad & Q \cap \bar D = \emptyset
\end{aligned}\right . 
$$
and put $ g = \chi _ { 4Q} ( f - \bar f _ { 4Q})$ and then  put $h =
f-g$. We let $v$ and $w$ solve the mixed problems
$$
\left \{\begin{aligned}
&\Delta v = 0 , \qquad & \mbox{in } \Omega\\
& v = 0 , \qquad & \mbox{on } D \\
& \frac { \partial v }{ \partial  \nu } =g , \qquad & \mbox{on } N \\
&\nabla v ^ * \in L^ 1 ( \partial \Omega)
\end{aligned}\right. 
\qquad
\left \{\begin{aligned}
&\Delta w = 0 , \qquad & \mbox{in } \Omega\\
& w = 0 , \qquad & \mbox{on } D \\
& \frac { \partial w }{ \partial  \nu } =h , \qquad & \mbox{on } N \\
&\nabla w ^ * \in L^ 1 ( \partial \Omega)
\end{aligned}\right.
$$
Since we have that $ L^p(N) \subset H^1(N)$ if $ p >1$, these
solutions exist thanks to Theorem 4.17.
}
Let $v$ and $w$ be defined as in section 7 of \cite{MR3042705}.  We 
set $ F_Q =
 M ((\nabla v ^*)^{1/2})^2$ and   $ R_Q =  M ((\nabla w ^*)^{1/2})^2$. 
  By uniqueness for the
$L^1$-mixed problem (Theorem 5.1), we have that $ u =v+w$ and thus it
  follows that we have (7.1). By Theorem 4.17, we have the estimate
$$
  \int_{ \partial \Omega } \nabla v ^ * \, d \sigma
  \leq C \|g\| _ { H^1 (N)} \leq \sigma ( 4Q ) ^ { 1-1/p}
  \left (\int _ { 4 Q} |f|^ p\, d\sigma \right) ^ { 1/p}.
$$
  From this and the Hardy-Littlewood maximal theorem, we obtain
\begin{equation}\label{7.2}
  \average _ {2Q}
  M((\nabla v ^*) ^ { \frac 1 2} ) ^2 \, d\sigma
  \leq C \left( \average _ {4 Q} |f|^p \, d\sigma
    \right) ^ { 1/p}
\end{equation}
which is (7.2).

Before estimating  $R_Q$, we give two technical lemmata.
  \begin{lemma} \label{Minf}
    Suppose that $x,y$ are in $ \partial \Omega$ and $ |x-y| < Ar$,
    then we have
    $$ M_\infty (f) (x) \leq C_A M_ \infty (f) (y).
    $$
  \end{lemma}
  \begin{proof} By the triangle inequality, we have $ \sball x s
    \subset \sball y { s+Ar}$. Thus it follows that
    $$
    \average _ { \sball x s } |f|\, d\sigma
    \leq  \frac { \sigma ( \sball y { s+Ar})}{\sigma (\sball x s )}
    \average_{ \sball y { s+Ar}} |f|\, d\sigma. 
    $$
    If we require that $s \geq r$, then
    we have a constant so that
    $ \sigma ( \sball y { s+Ar})/\sigma (\sball x s ) 
\leq C_A$ and the Lemma follows. 
  \end{proof}

  \begin{lemma} \label{Far} For $p>0$, we have 
    $$
    u ^ \vartriangle (x) \leq C M _ {\infty }( ( u^ * )^p)) ^
    { 1/ p }(x).
    $$
    The constant depends on the value of $p$ and the constant $c$
    entering into the definition of $ u ^ \vartriangle$. 
  \end{lemma}
\note{ If we truncate with $ \delta (y)$, rather than $|x-y|$, we
should   obtain an estimate with $ u ^ \vartriangle $ on the right.
  }
  \begin{proof} Fix $x\in \partial \Omega$ and suppose that $y \in
    \Gamma (x)$. Fix $ \hat y$ so that $ |y-\hat y| = d(y) = \dist(y,
    \partial \Omega)$ and observe that if $ |z-\hat y | <\alpha d(y)$, we
    have $y \in \Gamma (z)$. This implies $ |u(y) | \leq u^ * (z) $ 
    for $ z \in \sball { \hat y } { \alpha d(y)}$.
    By the triangle inequality $
    |x-\hat y| \leq |x-y | + |y - \hat y| \leq (2+ \alpha )
    d(y)$. Hence we have that $\sball { \hat y} {\alpha d (y)} \subset
    \sball x { ( 2 + 2 \alpha ) d(y)} $. It follows that
    $$
    |u(y) | \leq 
     \frac { \sigma ( \sball x { (2 + 2 \alpha )d(y)} )}
    { \sigma ( \sball {\hat y} { \alpha d(y) } ) } \average _ { \sball x { (
        2+ 2 \alpha ) d(y)}} u^ * \, d\sigma  . 
    $$
If we assume that $ |x-y| >cr$, then we have $ d(y)
> cr/(1+ \alpha)$ and   obtain
$$
u^ \vartriangle (x) \leq C M _ \infty ( u ^ *) ( x)
$$
which is our result with $ p=1$. To obtain the Lemma for other values
of $p$, apply the case with $p =1$ to $|u|^p$.
  \end{proof}
  
  To estimate $R_Q$, we will fix $x \in 2Q$ and $r > c \diam (Q)$ so that $ \sball x
  { 4r } \subset 4Q$. Our goal is to show that 
\begin{equation} \label{MainStep}
  \left (\average  _ { \sball x r } M ( ( \nabla w ^ * )^ {\frac  1 2
    } ) ^{2q} \, d\sigma \right ) ^ { 1/q}
  \leq
  C\left[ \average _{ \sball x {4r} } M (( \nabla u ^ * ) ^{ \frac 1 2 }
    )^2 \, d\sigma   + \left  ( \average _ { 4Q }  |f|^p \, d\sigma
    \right ) ^ { 1/p} \right]. 
\end{equation}
Covering $2Q$ by a finite collection of balls we may obtain (7.3) from
\eqref{MainStep}.

To obtain \eqref{MainStep}, we begin by observing that
  \begin{equation} \label{ZeroTerm}
  R_Q = M ((\nabla w ^*)^{\frac 12})^2
  \leq C ( M_ { \infty } (( \nabla w ^ * ) ^{ \frac 1 2} ) ^2
  +M_ { 0 } (( \nabla w ^ \triangledown ) ^{ \frac 1 2} ) ^2
    +M_ { 0 } (( \nabla w ^ \vartriangle ) ^{ \frac 1 2} ) ^2)
  \end{equation}
  and will proceed to estimate the three terms on the right of
  \eqref{ZeroTerm}. 

  From Lemma \ref{Minf}, we have
  \begin{equation}\label{OneTerm}
    \begin{aligned}
    \sup _ { y \in \sball x r } M_ \infty ( ( \nabla w ^  *)^ {\frac 12
    } ) ^ 2 (y) &\leq C \inf _ { y \in\sball x r }
    M_ \infty ( ( \nabla w ^  *)^ {\frac 12
    } ) ^ 2 (y) \\
&    \leq C \average _ {\sball x r } M _ \infty
    ( ( \nabla w ^  *)^ {\frac 12
    } ) ^ 2  \, d\sigma.
    \end{aligned}
  \end{equation}
  Next we observe that Lemma 6.1 and the Hardy-Littlewood maximal
  theorem gives that
\begin{equation} \label{TwoTerm}
\left (  \average _ { \sball x r} M_ 0(( \nabla w ^ \triangledown) ^ {\frac 1
    2} )^ { 2q} \, d\sigma \right ) ^ { 1/q}
  \leq C \left ( \average _ { \sball x { 4r} } \nabla  w ^ * \,
  d\sigma 
    + |\bar f _ { 4Q}|\right). 
\end{equation}

Finally, to estimate $ M_ 0 (( \nabla w ^\vartriangle) ^{ \frac 1
  2} ) ^ 2 $, we may use Lemma \ref{Minf} and Lemma \ref{Far} to see
that
$$
\nabla w ^ \vartriangle (y) \leq C M _ \infty (( \nabla w ^ * ) ^ {
  \frac 1 2 })^2 (z),\qquad
y \in \sball x { 2r}, \  z\in \sball x { 4r}.
$$
For $y\in \sball x r $, the maximal function $M_0(h)(y)$ depends only on
the values of $h$ in $ \sball x { 2r}$, thus we obtain
\begin{equation}\label{ThreeTerm}
\sup _ { y \in \sball x r } M_0 ( ( \nabla w^ \vartriangle)^{ \frac 1 2
})^2(y)
\leq C  \average _ { \sball x { 4r} }
  M ( ( \nabla w^*)^{ \frac 1 2
  })^2 \, d\sigma.
\end{equation}
  From (\ref{ZeroTerm}--\ref{ThreeTerm})  we conclude that
  $$
 \left (  \average _ { \sball x { 4r} }
  M ( ( \nabla w^*)^{ \frac 1 2
  })^{2q} \, d\sigma \right  ) ^ { 1/q}
  \leq 
  C\left[  \average _ { \sball x { 4r}} M((\nabla w ^ *) ^ {\frac 1 2}  )^2
      \, d\sigma +\left ( \average _ { 2Q} |f|^ p \, d\sigma \right ) ^ { 1/p}\right ].
    $$
To complete the proof of \eqref{MainStep}, we write $ w = u-v$ and then use \eqref{7.2} to
obtain
\begin{equation}
\begin{aligned}
\average _ { \sball x { 4r}} M((\nabla w ^ *) ^ {\frac 1 2}  )^2
  \, d\sigma
&  \leq C \left (
  \average _ { \sball x { 4r}} M((\nabla u ^ *) ^ {\frac 1 2}  )^2
    \, d\sigma +
    \average _ { \sball x { 4r}} M((\nabla v ^ *) ^ {\frac 1 2}  )^2
      \, d\sigma\right ) \\
&      \leq
    C \left(    \average _ { \sball x { 4r}} M((\nabla u ^ *) ^ {\frac 1 2}  )^2
          \, d\sigma +\left ( \average_{2Q} |f|^p  \, d\sigma \right ) ^ { 1/p}\right).
\end{aligned}
\end{equation}
          Our claim \eqref{MainStep} follows.

\bibliographystyle{plain}


\def\cprime{$'$}

\end{bibunit}

\small \noindent \today 
\end{document}